\documentclass[11pt]{amsart}
\usepackage{amscd,amssymb,graphics}

\usepackage{amsfonts}
\usepackage{amsmath}
\usepackage{amsxtra}
\usepackage{latexsym}
\usepackage[mathcal]{eucal}

\usepackage{graphics,colortbl}

\input xy
\xyoption{all}
\usepackage{epsfig}

\usepackage[pdftex,bookmarks,colorlinks,breaklinks]{hyperref}

\newtheorem{thm}{Theorem}[section]
\newtheorem{cor}[thm]{Corollary}
\newtheorem{lem}[thm]{Lemma}
\newtheorem{prop}[thm]{Proposition}

\newtheorem{question}[thm]{Question}

\theoremstyle{definition}
\newtheorem{defin}[thm]{Definition}

\theoremstyle{remark}
\newtheorem{remark}[thm]{Remark}
\newtheorem{remarks}[thm]{Remarks}

\newtheorem{examples}[thm]{Examples}

\numberwithin{equation}{section}

\numberwithin{equation}{section}

\newcommand{\delete}[1]{} 
\newcommand{\nt}{\noindent}

\def\eps{{\varepsilon}}
\newcommand{\sk}{\vskip 0.2cm}

\newcommand{\ben}{\begin{enumerate}}
	
	\newcommand{\een}{\end{enumerate}}
\newcommand{\bit}{\begin{itemize}}
	
	\newcommand{\eit}{\end{itemize}}

\def\eps{{\varepsilon}}

\def\al{\alpha}

\def\Om{\Omega}

\def\s{\sigma}
\def\ga{\gamma}

\def\a{\alpha}

\newcommand{\Del}{\Delta}

\def\Iso{\operatorname{Iso}}
\def\Asp{\operatorname{Asp}}
\def\RUC{\operatorname{RUC}}
\def\WAP{\operatorname{WAP}}

\def\SUC{\operatorname{SUC}}

\newcommand{\Tame}{$\mathrm{Tame}\ $}
\newcommand{\Sturm}{$\mathrm{Sturm}\ $}
\newcommand{\WRN}{$\mathrm{WRN}\ $}

\def\Sp{{\mathrm{Split}}}
\newcommand{\card}{\rm{card\,}}

\def\R {{\mathbb R}}
\def\N {{\mathbb N}}
\def\Z {{\mathbb Z}}

\def\T {{\mathbb T}}

\def\A{{\mathcal{A}}}

\newcommand{\lan}{\langle}
\newcommand{\ran}{\rangle}

\newcommand{\br}{\vspace{2 mm}}

\newcommand{\cls}{{\rm{cls\,}}}

\def\diam{{\mathrm{diam}}}
\def\Iso{{\mathrm{Iso}}\,}
\def\Aut{{\mathrm Aut}\,}

\def\wrt{with respect to }

\def\QED{\nobreak\quad\ifmmode\roman{Q.E.D.}\else{\rm Q.E.D.}\fi}

\begin{document}

\title[]{Circularly ordered dynamical systems} 

\author[]{Eli Glasner}
\address{Department of Mathematics,
	Tel-Aviv University, Ramat Aviv, Israel}
\email{glasner@math.tau.ac.il}
\urladdr{http://www.math.tau.ac.il/$^\sim$glasner}

\author[]{Michael Megrelishvili}
\address{Department of Mathematics,
	Bar-Ilan University, 52900 Ramat-Gan, Israel}
\email{megereli@math.biu.ac.il}
\urladdr{http://www.math.biu.ac.il/$^\sim$megereli}

\date{August 29, 2016}

\thanks{This research was supported by a grant of the Israel Science Foundation (ISF 668/13)}

\keywords{circular order, linear order, enveloping semigroup, tame dynamical system, Sturmian system, subshift, symbolic system, Rosenthal space}

\subjclass[2010]{Primary 37Bxx, 46-xx; Secondary 54H15, 26A45}


\begin{abstract}
We study 
topological properties of 
 circularly ordered dynamical systems and prove that every such system 
 is representable on a Rosenthal Banach space, hence, is also tame. 
We derive some consequences for topological groups. 
 We show that several Sturmian like symbolic  $\Z^k$-systems are circularly ordered. 
Using some old results 
we characterize circularly ordered minimal cascades.  
\end{abstract}

\maketitle

\setcounter{tocdepth}{1}
 \tableofcontents

\section*{Introduction} 

In this work we study circularly ordered dynamical systems. 
A {\it circular order} ({\it c-order} in short) on a set is, intuitively speaking, a linear order which has been bent into a ``circle". 
This concept has its roots at least in an old article of Huntington \cite{Hunt}. 
It has several applications in Geometry, Combinatorics,  Logic, Algebraic Topology and General Topology. 
 
We give some applications of c-order in Topological Dynamics. 
We say that a dynamical $G$-system $X$ is a {\it c-orderable} if $X$, as a topological space, is c-ordered 
and  for each $g \in G$, the corresponding translation $\tilde{g}: X \to X$ is 
c-order preserving (Definition \ref{d:COsystem}).  
We investigate the following question: 

\begin{question} \label{q:1} 
Which  compact (minimal) dynamical $G$-systems are c-ordered ? 
In particular, which symbolic $\Z^k$-systems are c-ordered ?  
\end{question} 

Note that (in contrast to the circular order case) every linearly ordered minimal $G$-space is trivial. Indeed every compact linearly ordered space has its minimum and maxumum elements,
the ``end points" of the system, which are necessarily fixed. 

%
As we will see  c-orderable dynamical systems are tame.
Tame dynamical systems were introduced by A. K\"{o}hler \cite{Koh}
and their theory was later developed in a series of works by several authors
(see e.g. \cite{Gl-tame,GM1,GM-rose,H,KL,Gl-str,Rom}).  
Recently, connections to other areas of mathematics like Banach spaces, coding theory,
substitutions and tilings, and even model theory and logic, were established 
(see e.g. \cite{Auj,Ibar, Ch-Si} and the survey \cite{GM-survey} for more details). 

We recall that 
a metric dynamical $G$-system $X$ is tame if and only if the enveloping semigroup $E(X)$ 
has cardinality at most $2^{\aleph_0}$ \cite{GM1,GM-AffComp}, 
iff every element $p \in E(X)$ of its enveloping semigroup is Baire class 1 function $p: X \to X$
\cite{GMU}. 
Other characterizations of tameness are given using the combinatorial notion of 
independence 
\cite{KL}. 
Finally, the metrizable tame systems are exactly those systems which admit 
a representation on a separable Rosenthal Banach space \cite{GM-rose}
(a Banach space is called \emph{Rosenthal} if it does not contain an isomorphic copy of $l_1$). 
For a survey of this theory we refer the reader to \cite{GM-survey}.

\br

Using tame families and families of functions with bounded variation we show that 
(even non-metrizable)
circularly ordered dynamical systems are representable on Rosenthal Banach spaces (Theorem \ref{t:CoisWRN}). 
Rosenthal representable dynamical systems are tame (see \cite{GM-rose}, or Theorems \ref{Tamecriterion} and \ref{WRNcriterion} below). 

Another approach to obtaining 
tameness 
is via the property of being null. 
For a general group action being null implies being tame.  
This is a result of Kerr and Li \cite{KL}
who used independence properties  to characterize these classes.
In Theorem \ref{null} we show easily that c-ordered systems are null, hence tame.

As we show below, many naturally defined symbolic systems, e.g. Sturmian-like systems, 
are circularly ordered
(Theorem \ref{t:multi}). It thus follows that these systems are tame. 
 
 \sk

Notation:
\ben 
\item As in \cite{GM-rose,GM-tame} we denote by 
\WRN (respectively, RN) the class of Rosenthal (respectively, Asplund) representable dynamical systems and by \Tame the class of 
tame systems, Definition \ref{d:tameDS}.  
(Recall that a Banach space is Asplund iff its dual has the Radon-Nikodym property.)
Within the class of metrizable systems, $\mathrm{WRN} = \mathrm{Tame}$. 
\item 
$\mathrm{LODS}$ and $\mathrm{CODS}$ will denote 
the classes of linearly ordered and  circularly ordered compact dynamical systems, respectively. Note that 
$\mathrm{LODS} \subset \mathrm{CODS}$. 
\item 
We will denote by  \Sturm the class of Sturmian dynamical systems. 
\een

\sk

$\mathrm{Sturm}$ is an important and intensively studied class of binary 
symbolic systems 
whose origins go back, at least, to the classical work of Morse and Hedlund \cite{MH}. 
The following inclusions are nontrivial: 

\centerline{\Sturm $\subset \mathrm{CODS} \subset$ \WRN $\subset \mathrm{Tame}$.}

Theorems \ref{t:multi} and \ref{t:CoisWRN} yield the first two inclusions. 
For the inclusion \WRN $\subset \mathrm{Tame}$, we refer to  \cite{GM-rose,GM-tame}. 
In particular, $\mathrm{LODS} \subset \mathrm{WRN}$. 
 We note that this latter fact gives a valuable information even
in the purely topological case (applying it to trivial identity actions), see Section \ref{s:topology}.  
We derive also some consequences for topological groups. In particular, Theorem \ref{t:GrRepr} shows that for every 
c-ordered compact space $X$ the topological group $H_+(X)$ of all c-order preserving homeomorphisms (equipped with the  compact open topology) is Rosenthal representable. 

In Section \ref{s:m} 
we show that several Sturmian like symbolic  $\Z^k$-systems are circularly orderable (and minimal). 
For the case $k=1$ this result gives a new proof of a result by Masui \cite{Masui}. 
 Multidimensional Sturmian like $\Z^k$-systems are studied, e.g.,
in \cite{BFZ,Fernique}.

Using old results concerning circle homeomorphisms (in particular, Markley's results about Denjoy's minimal systems, \cite{Ma}) we characterize, in Section \ref{s:MinC-Ord}, 
infinite circularly ordered metric minimal cascades.

\br

\section{Preliminaries}  

We use the notation of \cite{GM-rose, GM-survey}. 
By a topological space we mean a Tychonoff (completely regular Hausdorff) space. 
The closure operator in topological spaces
will be denoted by $\cls$.  
A function $f: X \to Y$ is \emph{Baire class 1 function} if the inverse image $f^{-1}(O)$ of every open set $O$ is
$F_\sigma$ in $X$. 
For a pair of topological spaces $X$ and $Y$, $C(X,Y)$ is the  
set of continuous functions from $X$ into $Y$. 
We denote by $C(X)$ the Banach 
algebra of {\em bounded} continuous real 
valued functions even when 
$X$ is not necessarily compact.

All semigroups $S$ are assumed to be monoids, 
i.e., semigroups with a neutral element which will be denoted by $e$.  
A (left) {\em action} of $S$ on a space $X$ is a map $\pi : S \times X \to X$ such that
$\pi(st,x) = \pi(s,\pi(t,x))$ for every $s, t \in S$ and $x \in X$. We usually simply write
$sx$ for $\pi(s,x)$.
Also actions are \emph{monoidal} (meaning $ex=x, \forall x \in X$). 

An $S$-\emph{space} is a topological space $X$ equipped with a 
continuous action $\pi: S \times X \to X$ of 
a 
topological semigroup $S$ on the space $X$. 
A compact $S$-space $X$ is called a \emph{dynamical $S$-system} and is denoted by $(S,X)$.
Note that in \cite{GM-rose} and \cite{GM-AffComp} we deal with the more general case of separately continuous  actions. We reserve the symbol $G$ for the case where $S$ is a topological group. 
As usual, a continuous map $\a : X \to Y$ between two $S$-systems is called an {\em $S$-map}
or {\em a homomorphism}  
when $\a(sx)=s\a(x)$ for every $(s,x) \in S \times X$.  

For every $S$-space $X$ we have a monoid homomorphism
$j: S \to C(X,X)$, $j(s)=\tilde{s}$,
where $\tilde{s}: X \to X, x \mapsto sx=\pi(s,x)$ is the {\em $s$-translation} ($s \in S$). 

The {\em enveloping semigroup} $E(S,X)$ (or just $E(X)$) 
for a compact $S$-system
is defined as the pointwise closure ${\cls}(j(S))$ 
of $\tilde{S}=j(S)$ in $X^X$. Then $E(S,X)$ 
is a compact right topological monoid (i.e. right multiplication
$q \mapsto qp, \  q \in E(S,X)$, is continuous for every $p \in E(S,X)$).

By a \emph{cascade} on $X$ we mean a $\Z$-action $\Z \times X \to X$.
When dealing with cascades we usually write $(X,\sigma)$, where $\sigma$ is the $s$-translation $X \to X$ corresponding to $s=1$ ($0$ acts as the identity), instead of $(\Z, X)$.

\subsection{Some classes of dynamical systems}
\label{s:classes} 


We are mainly interested in the class of tame dynamical systems.
For the history of this notion
we refer to
\cite{Gl-tame} and \cite{GM-rose}. 

\begin{defin} \label{d:tameDS} 
A compact dynamical $S$-system $X$ is said to be \emph{tame} if for every $f \in C(X,\R)$ the family $fS:=\{fs: s \in S\}$  has no $l_1$-subsequence 
(where, 
$(fs)(x):=f(sx)$); see Definition \ref{d:l1}.  
Equivalently, if $fS$ has no independent subsequence (see Theorem \ref{f:sub-fr}, Definitions \ref{d:l1} and \ref{d:ind}). 
\end{defin}

The following principal result is a dynamical analog
of the  Bourgain-Fremlin-Talagrand dichotomy \cite{BFT,TodBook}.

\begin{thm} \label{D-BFT}
	\cite{GM1} \emph{(A dynamical version of BFT dichotomy)}
	Let $X$ be a compact metric dynamical $S$-system and let $E=E(X)$ be its
	enveloping semigroup. 
Either
	\begin{enumerate}
		\item
		$E$ is a separable Rosenthal compact (hence $E$ is Fr\'echet and ${card} \; {E} \leq
		2^{\aleph_0}$); or
		\item
		the compact space $E$ contains a homeomorphic
		copy of $\beta\N$ (hence ${card} \; {E} = 2^{2^{\aleph_0}}$).
	\end{enumerate}
	The first possibility
	holds iff $X$ is a tame $S$-system.
\end{thm}

Thus, a metrizable dynamical system is tame iff
$\card(E(X)) \leq 2^{\aleph_0}$ iff $E(X)$ is a Rosenthal compactum (or a Fr\'echet space).
Moreover, by \cite{GMU} a metric $S$-system is tame iff every $p \in E(X)$
is a Baire class 1 map $p: X \to X$. By \cite{GM-rose, GM-AffComp}, 
a general (not necessarily, metrizable) compact $S$-system $X$ is tame 
iff every $p \in E(X)$ is a \textit{fragmented map}, Definition \ref{d:frag} (equivalently, Baire 1, when $X$ is metrizable). 
Basic properties and applications of fragmentability in topological dynamics can be found in \cite{GM-rose, GM-survey, GM-tame}.

Recall that a dynamical $S$-system $X$ is weakly almost periodic ($\mathrm{WAP}$) if and only if every $p \in E(X)$ is a continuous map. So, every WAP system is tame. The class of hereditarily nonsensitive systems ($\mathrm{HNS}$) is an intermediate class of systems, \cite{GM-survey}, $\mathrm{WAP} \subset \mathrm{HNS} \subset \mathrm{Tame}$. A metrizable $S$-system $X$ is HNS iff $(S,X)$ is RN iff $E(X)$ is metrizable, \cite{GM1,GMU}. 
The class of tame dynamical systems is quite large. 
It is closed under subdirect products and factors.

\subsection{Some classes of functions}

A \emph{compactification} of $X$ is a continuous map $\ga: X \to Y$  with a dense range where $Y$ is 
compact.  
When $X$ and $Y$ are $S$-spaces and $\ga$ is an $S$-map we say that $\ga$ is 
an \emph{$S$-compactification}. 

A function $f \in C(X)$ on an $S$-space $X$ is said to be Right Uniformly Continuous if the induced right action $C(X) \times S \to C(X)$ is continuous at the point $(f,e)$, where $e$ is the identity of $S$. Notation: $f \in \RUC(X)$. 
If $X$ is a compact $S$-space then $\RUC(X)=C(X)$.  
Note that $f \in \RUC(X)$ if and only if there exists an  
$S$-compactification $\ga: X \to Y$ such that $f= \tilde{f} \circ \ga$ for some $\tilde{f} \in C(Y)$. 
In this case we say that $f$ \emph{comes} from the $S$-compactification $\ga: X \to Y$. 
%


The function $f$ is said to be: a)
\emph{WAP}; b) \emph{Asplund}; c)  \emph{tame}
if  $f$ comes from an $S$-compactification $\ga: X \to Y$
such that $(S,Y)$ is: WAP, HNS or tame respectively.
For the corresponding classes of functions we use the notation:
$\WAP (X)$, $\Asp(X),$ $\mathrm{Tame}(X)$, respectively. Each of these is a
norm closed
$S$-invariant subalgebra of the $S$-algebra
$\RUC(X)$
and
$\WAP(X) \subset \Asp(X) \subset \mathrm{Tame}(X).$
For more details see \cite{GM-AffComp,GM-survey}. 
 As a particular case 
we have defined the algebras
$\WAP(S)$, $\Asp(S)$, $\mathrm{Tame}(S)$ corresponding to the left action of $S$ on $X:=S$.

	\subsection{Symbolic systems}
	
	The binary \emph{Bernoulli shift system} is defined as the cascade $(\Omega,\s)$, where $\Omega:=\{0,1\}^{\Z}$.
	We have the natural
	$\Z$-action on the compact metric space  
	$\Omega$ induced by the $\s$-shift:
	$$
	\Z \times \Omega \to \Omega, \ \ \ \s^m (\omega_i)_{i \in \Z}=(\omega_{i+m})_{i \in \Z} \ \ \ \forall (\omega_i)_{i \in \Z} \in \Omega, \ \ \forall m \in \Z.
	$$
	
	More generally, for a
	discrete group $G$ and a finite alphabet $\Delta$ 
	the compact space $\Delta^G$ is a compact $G$-space under the action
	$$G \times \Delta^G \to \Delta^G, \ \ (s \omega)(t)=\omega (ts),\ \omega \in \Delta^G, \ \ s,t \in G.$$
	A closed $G$-invariant subset $X \subset\Delta^G$ defines a
	subsystem $(G,X)$. Such systems are called {\em subshifts\/} or
	{\em symbolic dynamical systems\/}. 
%
%
%
%
	

\subsection{Coding functions}


\begin{defin} \label{d:tametype1} \
	\begin{enumerate}
		\item
		Let $G \times X \to X$ be an action on a (not necessarily compact) space
		$X$,
		$f: X \to \R$ a bounded (not necessarily continuous)
		function, and $z \in X$. Define a \emph{coding function} as follows:
		$$
		\varphi:=m(f,z) : G \to \R, \ g \mapsto f(gz).
		$$ 
		\item
		When $G=\Z^k$ and $f(X)=\{0,1,\dots,d\}$ we say
		that $f$ is a  
		\emph{$(k,d)$-code}. Every such code generates a point transitive subshift $G_{\varphi}$ of $\Delta^G$,
		where $\Delta=\{0,1, \dots, d\}$ and 
		$$
G_{\varphi} : = \cls_p \{g \varphi : g \in G\} \subset  \Delta^G \ \ \ \ (\text{where} \ g \varphi(t)=\varphi(tg))
$$ 
is the pointwise closure of the left $G$-orbit $G \varphi$ in the space $\{0,1, \cdots, d\}^G$.
		
		\item 
		In the particular case where
		$\chi_D: X \to \{0,1\}$
		is the characteristic function 		
		of a subset $D \subset X$ and $G=\Z$, we get a $(1,1)$-code.
	\end{enumerate}
\end{defin}

\sk 

Regarding some dynamical and combinatorial aspects of coding functions 
see \cite{Fernique,BFZ}. 

\begin{question} \label{q:coding} 
	When is a coding function  $\varphi$ tame ? 
	Equivalently, 
	when is the associated transitive subshift system $G_{\varphi} \subset \{0,1\}^{\Z}$, with $\varphi =	m(D,x_0)$  tame? 
	When are such subshifts c-ordered ? 
\end{question}


\begin{remarks} \label{r:codes} \ 
\ben 
\item Some restrictions on $D$ are really necessary because \emph{every} binary bisequence 
\newline $\varphi: \Z \to \{0,1\}$ can be encoded as $\varphi=	m(D,x_0)$. 
\item It follows from results in \cite{GM-rose} that a coding bisequence
$c: \Z \to \R$ is tame iff it can be represented
as a generalized matrix coefficient of a Rosenthal Banach space representation.
That is, iff there exist: a Rosenthal Banach space $V$, a linear isometry $\s \in \Iso(V)$ and
two vectors $v \in V$, $\phi \in V^*$ such that
$$
c_n=\langle \s^n(v),\varphi \rangle = \phi(\s^n(v)) \ \ \ \ \forall n \in \Z.
$$ 
\een
\end{remarks}

\sk 
Recall (see for example \cite{BFZ}) that a bisequence $\Z \to \{0,1\}$ is \emph{Sturmian} if it is recurrent and has the minimal complexity $p(n)=n+1$.

\begin{defin} \label{d:StCode} 
	Let $P_0$ be the set
	$[0, t)$ and $P_1$ the set $[t, 1)$; let $z$ be a point in $[0, 1)$ 
	(identified with $\T$) via the rotation $R_{\a}$ we get the binary bisequence 
	$$\varphi: \Z \to \{0,1\}, n \mapsto \varphi(n)=s_n,$$ 
	by $s_n = 0$ when $R_{\a}^n(z)  \in P_0, s_n = 1$ otherwise.
	Equivalently $\varphi$ is defined as 
	$m(\chi_D,z)$ for $D:=[t,1)$. 
	These are called \emph{Sturmian like codings}.
	With $D=[1 - \a,1)$ we get the classical \emph{Sturmian bisequences}. 
	About a realization of the corresponding subshift see \cite{GM1} (and Example \ref{Sturm1} below). 
	For example, when $\a:=\frac{\sqrt{5}-1}{2}$ and $c = 1 -\a$ the corresponding sequence,
	computed at $z =0$, is called the \emph{Fibonacci bisequence}.
	
	Every Sturmian bisequence $\varphi$ induces a minimal symbolic system $\Z_{\varphi} \subset \{0,1\}^{\Z}$, which is said to be a \emph{Sturmian dynamical system}. Similarly one defines  \emph{Sturmian like systems}.  
\end{defin}

%

\br

\section{C-ordered topological spaces and systems}

First we give  
one of the most conventional axiomatic for the concept of circular ordering. 

\begin{defin} \label{newC} \cite{Kok,Cech} 
	Let $X$ be a set. A ternary relation $R \subset X^3$ on $X$ is said to be a {\it circular} (or, sometimes, \emph{cyclic}) order  
	if the following four conditions are satisfied. It is convenient sometimes to write shortly $[a,b,c]$ instead of $(a,b,c) \in R$. 
	\ben
	\item Cyclicity: 
	$[a,b,c] \Rightarrow [b,c,a]$;  
	
	\item Asymmetry: 
	$[a,b,c] \Rightarrow (a, c, b) \notin R$; 
		
	\item Transitivity:    
	$
	\begin{cases}
	[a,b,c] \\
	[a,c,d]
	\end{cases}
	$ 
	$\Rightarrow [a,b,d]$;
	
	\item Totality: 
	if $a, b, c \in X$ are distinct, then \ $[a, b, c]$ $\vee$ $[a, c, b]$. 
	\een
\end{defin}


\begin{lem} \label{l:property} For every c-order on $X$ we have:
	\ben 
	\item  $[a,b,c]$ implies that $a,b,c$ are distinct.  
%
	
	\item $
	\begin{cases}
	[c,a,x] \\
	[c,x,b]
	\end{cases}
	$ 
	$\Rightarrow [a,x,b]$. 
	\een
\end{lem}

\sk

For distinct $a,b \in X$ define the (oriented) \emph{intervals}: 
$$
(a,b)_R:=\{x \in X: [a,x,b], \ \ \ \  [a,b]_R:=(a,b) \cup \{a,b\}. 
$$

Sometimes we drop the subscript when context is clear.

\begin{prop} \label{Hausdorff}  \
	\ben 
	\item \cite[p. 6]{Kok} 
	For every c-order $R$ on $X$ the family of intervals 
	$$
	\{(a,b)_R : \ \  a,b \in X\}
	$$
	forms  a base for a topology $\tau_R$ on $X$ which we call the \emph{interval topology} of $R$. 
	
	\item The topology $\tau_R$ of every circular order $R$ is Hausdorff. 
	\een 
\end{prop} 
%
%
%
%
%
%

The prototypical example of a c-order the usual 
(counter-clockwise) circular ordering on the circle $\T$. Identify
$\T$, as a set, with $[0,1)$ and
define a ternary relation $R \subset [0,1)^3$ as follows
$$(x, y, z) \in R \Leftrightarrow (y-x)(z-y)(z-x) > 0.$$ 
Its topology $\tau_R$ gives the usual topology on $\T$. 

By a linear order $<$ on $X$ we mean a transitive relation which is totally ordered, meaning that for distinct $a,b \in X$ 
we have exactly one of the alternatives: $a<b$ or $b<a$. As usual, $a \leq b$ will mean that $a < b \vee a=b$.    

For every linearly ordered set $(X,<)$ the rays $(a,\to)$, $(\leftarrow,b)$, with $a,b \in X$, form a subbase for the
standard \emph{interval topology} $\tau_{<}$ on $X$.
It is well known that the interval topology is always Hausdorff (and even normal). 
A topological space is said to be \emph{Linearly Ordered Topological Space} ($\mathrm{LOTS}$) if its topology is $\tau_{<}$ for some linear order $<$. 
Similarly, a topological space is said to be \emph{circularly ordered topological space}  ($\mathrm{COTS}$) 
if its topology is $\tau_{R}$ for some circular order $R$. 


\begin{remark} \label{r:chech} \ \cite[page 35]{Cech}
	\ben 
	\item 
	Every linear order $<$ on $X$ defines a \emph{standard circular order} $R_{<}$ on $X$ as follows: 
	$[x,y,z]$ iff one of the following conditions is satisfied:
	$$x < y < z, \ y < z < x, \  z < x < y.$$	
	\item (cuts) Let $(X,R)$ be a c-ordered set and $z \in X$. 
	For every $z \in X$ the relation 
	$$z <_z a, \ \ \ \  a <_z b \Leftrightarrow [z,a,b] \ \ \ \forall a \neq b \neq z \neq a$$
	 is a linear order on $X$ and $z$ is the least element. This linear order 
	 restores the original circular order. Meaning that $R_{<_z}=R$. 
	\een
\end{remark}

On the set $\{0, 1, \cdots, n-1\}$ consider the standard c-order modulo $n$. 
Denote this c-ordered set, as well as its order, simply by $C_n$. 
Every finite c-ordered set with $n$ elements is isomorphic (Definition \ref{c-ordMaps}) to $C_n$. 

\begin{defin} \label{d:cycl}  Let $(X,R)$ be a c-ordered set. 
	We say that a vector
	$(x_1,x_2, \cdots, x_n) \in X^n$ is a \textit{cycle} in $X$ if it satisfies the following two conditions: 
	
	\ben 
	\item For every $[i,j,k]$ in $C_n$ and \textit{distinct} 
	$x_i, x_j, x_k$ we have $[x_i,x_j,x_k]$; 
	\item $x_i=x_k \ \Rightarrow$ \ 
	$(x_i=x_{i+1}= \cdots =x_{k-1}=x_k) \ \vee \ (x_k=x_{k+1}=\cdots =x_{i-1}=x_i).$  
	\een  
	\textit{Injective cycle} means that all $x_i$ are distinct. 
\end{defin}

\begin{defin} \label{c-ordMaps} Let $(X_1,R_1)$ and  $(X_2,R_2)$ be c-ordered sets. 
	A function $f: X_1 \to X_2$ is said to be {\it c-order preserving}, 
	or {\it COP}, if $f$ moves every cycle to a cycle. 
	Equivalently, if it satisfies the following two conditions:
	
	\ben 
	\item For every $[a,b,c]$ in $X$ and \textit{distinct} 
	$f(a), f(b), f(c)$ we have $[f(a), f(b), f(c)]$; 
	\item If $f(a)=f(c)$ then $f$ is constant on one of the closed intervals $[a,c], [c,a]$. 
	\een 
	\sk 
We let $M_+(X_1,X_2)$ be the collection of c-order preserving
maps from $X_1$ into $X_2$.

	$f$ is an \textit{isomorphism} if, in addition, $f$ is a bijection. 
Denote by $H_+(X)$ the group of all COP isomorphisms $X \to X$  (necessarily homeomorphisms). 
\end{defin}

A composition of c-order preserving maps is c-order preserving. 

\begin{examples} \label{ex:cop}  \
	\ben 
	\item 	A function $f: C_n=\{1,2, \cdots, n\}: \to X$ is COP if and only if the corresponding vector $(f(x_1),f(x_2), \cdots, f(x_n))$ is a cycle in $X$ (Definition \ref{d:cycl}). 	
\item Gluing all points of a given 
closed interval on a c-ordered set defines a COP map. 
\item In particular, gluing end-points $a,b$ of a \textit{gap interval} (i.e., $(a,b)$ or $(b,a)$ is empty) of any c-ordered set induces a COP map. 		
					\item For example, the projection $q: X(c) \to X$ (from Proposition \ref{cover}) 
					is COP (e.g., $[0,1] \to \T$). 
\item Every COP map $f: X \to C_{d+1}$ can be interpreted as a standard coloring function from Definition \ref{d:ColFun}. 			
	\een 
\end{examples}


%

\begin{defin} \label{d:COsystem} 
	We say that	a compact $S$-system $(X,\tau)$ is {\it circularly orderable} if there exists a compatible circular 
	 order $R$ on $X$ such that $X$ is $\mathrm{COTS}$ and 
	every $s$-translation $\tilde{s}: X \to X$ is COP. 
We denote by $\mathrm{CODS}$ the class of all c-orderable systems.  Similarly we have the class $\mathrm{LOTS}$ of all linearly ordered compact $S$-systems. 
\end{defin}



For every linearly (circularly) ordered compact space $X$ and every topological subgroup $G \subset H_+(X)$, with its compact open topology, the corresponding action $G \times X \to X$ defines a linearly (circularly) ordered $G$-system. 


\begin{prop} \label{inclusion} 
	$\mathrm{LODS} \subset \mathrm{CODS}$ and $\mathrm{LOTS} \subset \mathrm{COTS}$. 
	More precisely: 
	every compact linearly ordered space ($S$-system) is a circularly ordered space ($S$-system) \wrt the canonically associated circular order. 
\end{prop}
\begin{proof}
	Let $<$ be a linear order on $X$ such that the interval topology $\tau_{<}$ is compact. 
	Consider the canonical circular order (Remark \ref{r:chech}) $R:=R_{<}$ on $X$ and the corresponding topology $\tau_R$. Then 
	$$
	\tau_R \subseteq \tau_{<}.
	$$
	Indeed, it is enough to show that 
	$$
	(a,b)_R:=\{x \in X: [a,x,b]\}  \in \tau_{<}
	$$
for every distinct $a,b \in X$. 
We have two cases:

\begin{enumerate}
	\item $a<b$. Then $(a,b)_R=(a,b)_{<}.$ 
	\item $b<a$. Then $(a,b)_R=	(\leftarrow,b) \cup (a,\to)$. 
\end{enumerate}

In each case $(a,b)_R \in \tau_{<}$. 
    
	On the other hand, $\tau_R$ is Hausdorff by Proposition \ref{Hausdorff}. Since $\tau_{<}$ is compact we get 
	$$
	\tau_{<} = \tau_R.
	$$ 
	Finally, note that if an $s$-translation $X \to X$ is linear order preserving then it is also c-order preserving. 
\end{proof}

The compactness of $\tau_{<}$ is essential. Indeed, the linearly ordered set $X=[0,1)$ with respect to the c-order topology is in fact the circle. This gives a justification of the standard identification of the \emph{sets} $\T$ and (c-ordered) $[0,1)$. 

The circle $\T$ is clearly a factor space of a (linearly ordered) closed interval $[a,b]$ after identifying the endpoints. 
The following result shows that 
we have a similar situation for any c-ordered compact space. 
 
\begin{prop} \label{cover} 
	Let $(X,R)$ be a circularly ordered space. Then for every $c \in X$ there exists a linearly ordered space $X(c):=([c^-,c^+],<)$ such that $X$ is homeomorphic to the factor-space of
$X(c)$ identifying the endpoints of $X(c)$. Moreover, 
the corresponding quotient map $q: X(c) \to X$ (sometimes denoted by $q_c$) is closed and c-order preserving. When $X$ is compact then so is $X(c)$. 
	 \end{prop}
	 \begin{proof}
	 	Take a point $c \in X$ and consider the cut at $c$ where $c$ becomes the minimal element.  
	 	 Denote it by $c^{-}$. 
	 	Then we get a linearly ordered set $X$ by declaring $x < y$ whenever $(c,x,y) \in R$ 
	 	(see Remark \ref{r:chech}). 
	 	Adding to $X$ a new point $c^{+}$ as the greatest element we get a linearly ordered set $X(c)=[c^-,c^+]=X\cup \{c^+\}$ and a natural onto map $$q: X(c) \to X, \ q(c^{-})=q(c^{+})=c, \ q(x)=x \ \forall x \in (c^{-}, c^{+}).$$ 
	 	It is easy to see that $q$ is a closed continuous map (hence, a quotient map) \wrt the interval topologies. Moreover, $q$ is c-order preserving.  
	 	\sk 
	 	\textit{Continuity}: We have to show that $q$ is continuous at every point $z \in X(c)$. Let $U=(a,b)_R$ be an open basic neighborhood of $q(z) \in X$, where $a \neq b$.  There are two cases: 
	 	
	 	1) $q(z)\neq c$. 
	 	
	 	\nt Then $q(z)=z$. We can suppose that $c \notin U$ ($\tau_R$ is Hausdorff). In this case 
	 	we have necessarily $[c,a,z], [c,z,b]$. By the transitivity, $[c,a,b]$. This means that $a<_c b$. 
	 	Take $V:=(a,b)_{<_c}$, an open interval of $z$ in LOTS $X(c)$. Then $q(V) \subset U$. Indeed, $a <_c x <_c b$ implies that $[c,a,x], [c,x,b]$. Then 
	 	$[a,x,b]$ by Lemma \ref{l:property}. 
	 	This means that $q(x)=x \in (a,b)_R$.   
	 	
	 	2) $q(z)=c$. 
	 	
	 	\nt Then $z=c^-$ or $z=c^+$. We have $[a,c,b]$. Equivalently, $[c,b,a]$. So, $b <_c a$. Take 
	 	$$V:=(\leftarrow,b) \cup (a,\to).$$ Then $V$ is a neighborhood of $z$ in the LOTS $X(c)$ and $q(V) \subset U:=(a,b)_R$. Indeed:
	 	
	 	a) If $x <_c b$ then we have  $[c,x,b]$. Equivalently, $[b,c,x]$. 
	 	Since $[a,c,b]=[b,a,c]$ the transitivity axiom implies $[b,a,x]$. So we get $[a,x,b]$. 
	 	
	 	b) If $a <_c x$. Then $[c,a,x]$. Since $[c,b,a]$, by Lemma \ref{l:property} we get $[b,a,x]$. 
	 	 Which is equivalent to $[a,x,b]$. 
	 
%
%
%
%
%
%
%
%
%

	 \sk
	 
	 If $X$ is compact then 
	  $X(c):=[c^-,c^+]$ is compact, as a particular case of $X_A$ from Lemma \ref{c-doubling}.2 with the singleton $A:=\{c\}$. In this case the closedness of $q$ is clear because $X$ is Hausdorff (Proposition \ref{Hausdorff}). 
In fact, $q$ is closed in general, even when $X$ is not compact. We omit the details.  	
	 	\end{proof}

Every c-ordered space is a normal space. Indeed, recall that every $\mathrm{LOTS}$ is normal and the normality is preserved by closed maps onto Hausdorff spaces. Now Proposition \ref{cover}  finishes the proof. 


 \subsection{Splitting points and a construction of c-ordered spaces}
 
 We describe a general method for producing c-ordered $G$-systems. 


\begin{lem} \label{c-doubling}
	Let $R$ be a circular order on a set $X$ and $A \subset X$.
	\begin{enumerate}
		\item  There exist a canonically defined circularly ordered set $X_A = \Sp(X; A)$ and a continuous c-order preserving onto map $\nu: X_A \to X$ such that the preimage $\nu^{-1}(a)$ of any $a \in A$ consists of exactly two points and $\nu^{-1}(x)$ is a singleton for every $x  \in X \setminus A$. 
		
		\item If $X$ is compact then $X_A$ is also compact. 
		If, in addition, $A$ is countable and $X$ is metrizable then $X_A$ is metrizable. 
		\item The c-order on $X_A$ is uniquely defined. Let $\ga: M \to X$ be a c-ordered preserving map and $A \subset X$ such that the preimage $\ga^{-1}(a)$ of any $a \in A$ consists of exactly two points and $\ga^{-1}(x)$ is a singleton for every $x  \in X \setminus A$. Then $M$, as a c-ordered set, is canonically isomorphic to $X_A$. 	
		\item If $X$ is a c-ordered $G$-space with discrete group $G$ and $A$ is a $G$-invariant subset of $X$ then:
		\ben 
		\item 
		 the original action of $G$ on $X$ induces a natural continuous action on $X_A$ such that 
  $\nu: X_A \to X$ is a $G$-map.
  \item
  An inclusion of $G$-invariant subsets $A_1 \subset A_2$ 
  of $X$ induces a natural continuous onto $G$-map $\eta: X_{A_2} \to X_{A_1}$ such that $\nu_1 \circ \eta = \nu_2$.
			\item 
			Assume, in addition, that every point of $A$ is a limit point of $X$ from both sides. Then 
			$X_A$ is a minimal dynamical $G$-system iff $X$ is a minimal $G$-system. 	 
			\een 	
	\end{enumerate}

\end{lem}
\begin{proof}  (1) 
	$X_A$, as a set, is $\{a^{+}, a^{-}: a \in A \} \cup (X \setminus A)$. 
	$$\nu: X_A \to X, \  \  \ a^{\pm} \mapsto a, \ x \mapsto x \ \ \forall \ a \in A \ \forall x \in X \setminus A \ .$$ 
	Define a natural circular order on $X_A$ by the following two rules: 
	
	$\bullet$
	$[a,b,c]$ for every $(a,b,c) \in X_A^3$ where $[\nu(a),\nu(b),\nu(c)]$ in $X$. 
	
		$\bullet$ \ \ 
	$[a^{-}, a^{+}, u]$ \ \ 
	$[a^{+}, u, a^{-}]$ \ \ 
	$[u,a^{-}, a^{+}]$ \ \  for every $a \in A, u \notin \{a^+,a^-\}$. 
	\sk 
	The verification of the following claims are straightforward. 
	\sk
\nt 	Claim 1: $R_A$ is a c-order  
	on the set $X_A$.
	
	\sk
	
\nt 	Claim 2:  (The topology of $R_A$)   A (standard) base for the circular topology $\tau(R_A)$ on $X_A$ of $R_A$ 
at a point of the form $x \in X \setminus A$, 
	is the collection of sets 
	$$\nu^{-1}(u,v),\ \ \ x \in (u,v), \ u,v \in X.$$ 
	For $s^{-} \in X_A$ a
	basis will be the collection of sets of the form 
	$$(u,s^+)=\{s^{-}\} \cup \nu^{-1}(u,s).$$ 
	For
	$s^{+} \in X_A$ a basis will be the collection of sets of the form
	$$(s^-,v)=\{s^{+}\} \cup \nu^{-1}(s,v).$$

		\sk
	\nt 	Claim 3: $\nu: X_A \to X$ is a continuous c-order preserving map. 
		\sk
	
	(2) If $X$ is compact then one may apply Alexander's subbase theorem to show that $X_A$ is compact. 
Here we use the standard base of $\tau(R_A)$ constructed in Claim 2. 
	
	\sk
	
%
%
%
%
%
%
%
	
	(3) Let $\ga^{-1}(a)=\{a_1, a_2\}$ and $[a_1,a_2,u]$ for some $u \in M$. Then we claim that necessarily $[a_1,a_2,v]$ for every $v \notin \{a_1, a_2\}$ (this will imply that the c-order on $M$ is uniquely defined). If not then $[a_2,a_1,v]$ for some $v \notin \{a_1, a_2\}$. Since $\ga$ is COP we can suppose that $\ga(v) \neq \ga(u)$. 
	Using the cyclicity, $[a_2,a_1,v]=[a_1,v,a_2]$.  Together with $[a_1,a_2,u]$ we get by Lemma \ref{l:property} that $[v,a_2,u]$. On the other hand by the Transitivity axiom (for $[a_1,v,a_2]$ and $[a_1,a_2,u]$) we have $[a_1,v,u]$. Since $\ga: M \to X$ is COP (and $\ga(v),a=\ga(a_1)=\ga(a_2),\ga(u)$ are distinct) we obtain $[\ga(v),a,\ga(u)]$ and $[a,\ga(v),\ga(u)]$. 
	However, this is impossible by the Asymmetry axiom.

	(4) 
	(a) The induced action is given by
	$$G \times X_A \to X_A, \ g (s^+)=(gs)^+, 
	g (s^-)=(gs)^-, g(x)=gx \ \forall s \in A, \ \forall x \notin A.$$

	(b) and (c) are straightforward.  
\end{proof}

\begin{remarks} \label{r:add} \ 
	\ben 
	\item The $G$-space $X_A$ from Lemma \ref{c-doubling} item (4) sometimes will be denoted by $X_A = \Sp(X,G; A)$. 
	In the particular case of the cyclic group $G:=\lan R_{\al} \ran$ generated by an irrational angle $\al$ we use the notation $\Sp(\T,R_\al; A)$. 

	\item Some examples: \label{Sturm1} \ 
	\ben 
	\item 
	Note that if the subset $A \subset X$ is empty then $X_A=X$. 
	\item For the c-ordered circle $X:=\T$ with $A=\T$  we get the ``double circle" of Ellis, \cite{Ellis}, which we denote by $\T_{\T}$. 
	\item
	Another important prototype of Lemma \ref{c-doubling} is \cite[Example 14.10]{GM1}. It gives a concrete realization of the Sturmian subshift, Definition \ref{d:StCode} (with $t = 1 - \al$). 
	In this case the corresponding cascade is $\T_A$ with $A:=\{m\al, n(1-\al) : m, n \in \Z\}$. 
	\een 
	
	\item \label{r:prod} 
	For every c-ordered set $K$ and a linearly ordered set $L$ one may define the so-called 
	\emph{c-ordered lexicographic product} $\T \times L$. See for example \cite{CJ} and also Figure 1 below.  
	The 2-circle $\T_{\T}$ as the c-ordered set is the \emph{c-ordered lexicographic product} $\T \times \{-,+\}$. 	
	 Every splitting space $X_A$ in Lemma \ref{c-doubling} is a c-ordered subset 
	of $\T \times \{-,+\}$. 
	
	\begin{figure}[h]
		\begin{center} \label{F1}
			\scalebox{0.3}{\includegraphics{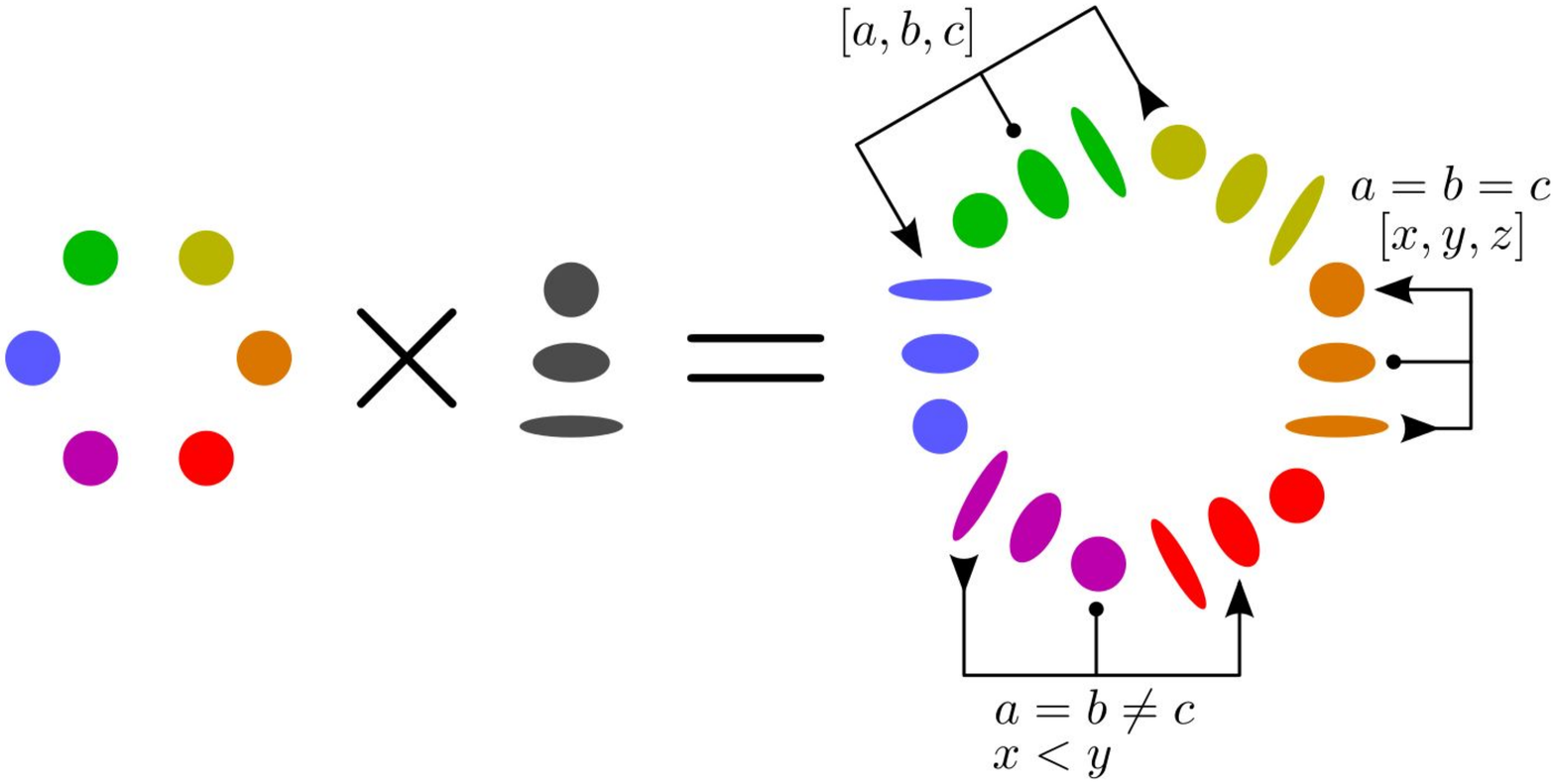}}
			\caption{c-ordered lexicographic product (from Wikipedia)}
		\end{center}
	\end{figure}

	\item The splitting points construction has its roots in linear orders. 
For every linearly ordered space $X$ and every subset $A \subset X$ one may define a new linearly ordered space $X_A$ and a continuous order preserving onto map $X_A \to X$. In particular, we can take $A=X$. Then  one gets a generalization of the double arrow space. 
See, for example, \cite{AH}. 
For $X=A=[0,1]$ the corresponding doubling procedure gives $[0,1] \times \{0,1\}$ with the lexicographic linear order. Removing two isolated points $0^-$ and $1^+$ we get 
the classical double arrow space which is $X_A$ with $X=[0,1]$ and $A=(0,1)$. 
Note that the ``double circle" $\T_{\T}$ 
is homeomorphic to the double arrow. This space $\T_{\T}$,  topologically, is embedded into the enveloping semigroup $E(X)$ of a Sturmian cascade generated by an irrational rotation $T: X \to X$ as in Definition \ref{d:StCode}. In fact, $\T_{\T}=E(X) \setminus \{\s^n: n \in \Z\}$. Moreover, 
in this case $E(X)$ is c-ordered and the embedding $\T_{\T} \hookrightarrow E(X)$ is c-order preserving. 
See Section \ref{s:E}.   
\een
\end{remarks}

\br

\subsection{C-ordered systems are null, hence also tame}

We begin by recalling the definition
of topological sequence entropy for a general dynamical system $(G, X)$.
Let $A=\{a_0, a_1, \dots\}$ be a sequence of of elements of $G$.
Given an open cover $\mathcal{U}$, let $N(\mathcal{U})$ be the minimal cardinality
of a subcover of $\mathcal{U}$. Define
$$
h^A_{top}(X,\mathcal{U})=\limsup_{n\to\infty} \frac{1}{n}\log N\left(\bigvee
_{i=0}^{n-1} {a_i}^{-1}(\mathcal{U})\right)
$$
The {\it topological entropy along the sequence $A$} is then defined by
$$
h^A_{top}(G, X)= \sup
\{h^A_{top}(X,\mathcal{U}) :  \mathcal{U}\ \text{an open cover of $X$}\}.
$$
A dynamical system $(G,X)$ is called {\em null}
if $h^A_{top}(G, X) =0$ for every sequence $A \subset G$.


By results of Kerr and Li \cite{KL05,KL} every null system is tame. 

\begin{thm}\label{null}
	A c-ordered metrizable cascade $(G,X)$ is null, hence also tame. 
\end{thm}

\begin{proof}
	It is easy to see that with no loss of generality we can work with open covers
	$\mathcal{U}$ consisting of a finite collection of open intervals.
	Let $\mathcal{U}$ be such a cover of cardinality $k$. 
	Now clearly, given $A=\{a_0<a_1<\ldots\}$, and denoting $N_n =
	N(\bigvee_{i=0}^{n-1}{a_i}^{-1}(\mathcal{U}))$, we have
	$$
	N_{n+1} =   N\left(\bigvee_{i=0}^{n}{a_i}^{-1}(\mathcal{U})\right)
	= N \left({a_i}^{-1} \mathcal{U} \vee \bigvee_{i=0}^{n-1}{a_i}^{-1}(\mathcal{U})\right)
	\le N_n + 2k.
	$$
	Therefore, $N_n \le 2kn$, whence
	$$
	h^A_{top}(X,\mathcal{U})=\limsup_{n\to\infty} \frac{1}{n} \log N_n =0.
	$$
\end{proof}



\sk

\section{Tame families of functions}

	\sk 
	
\subsection{Representations on Rosenthal spaces} 

\begin{defin}\label{d:l1}
	Let $f_n: X \to \R$ be a uniformly bounded sequence of functions on a \emph{set} $X$. Following Rosenthal we say that
	this sequence is an \emph{$l_1$-sequence} on $X$ if there exists a real constant $a >0$
	such that for all $n \in \N$ and choices of real scalars $c_1, \dots, c_n$ we have
	$$
	a \cdot \sum_{i=1}^n |c_i| \leq ||\sum_{i=1}^n c_i f_i||.
	$$
\end{defin}


For every $l_1$-sequence $f_n$ its closed linear span in $l_{\infty}(X)$
is linearly homeomorphic to the Banach space $l_1$.
In fact, the map $$l_1 \to l_{\infty}(X), \ \ (c_n) \to \sum_{n \in \N} c_nf_n$$ is a linear homeomorphic embedding.

A Banach space $V$ is said to be {\em Rosenthal} if it does
not contain an isomorphic copy of $l_1$. 
Every Asplund (in particular, every reflexive) space is Rosenthal. 
Recall that a dynamical system $(S,X)$ is WRN 
means that it is representable on a Rosenthal Banach space.
 

%

\begin{defin} \label{d:frag}
Let $X$ be a topological space and $(Y,\mu)$ is a uniform space.  
We say that a  function $f : X \to (Y,\mu)$ is
\emph{fragmented},
if for every $\eps \in \mu$ and every nonempty closed subset $A$ of $X$, 
 there exists a non-void relatively open subset $O \subset A$
such that $f(O)$ is $\eps$-small. 
\end{defin}

Write $\mathcal{F}(X)$ for the collection of real valued fragmented functions on $X$.
We note that when $X$ is a compact metrizable space, $f: X \to \R$ is 
fragmented iff it is of Baire class 1 iff $f$ is a pointwise limit of a sequence of continuous functions.

\begin{defin} \label{d:Ros-F} \cite{GM-rose}
Let $X$ be a topological space. We say that a subset $F\subset
C(X)$ is a \emph{Rosenthal family}
(for $X$) if $F$ is norm bounded and
the pointwise closure ${\cls_p}(F)$ of $F$ in $\R^X$ consists of fragmented maps,
that is,
${\cls_p}(F) \subset {\mathcal F}(X)$.
\end{defin}


\begin{defin}\label{d:ind}
	A sequence $f_n$ of
	real valued functions on
	a set
	$X$ is said to be \emph{independent} if
	there exist real numbers $a < b$ such that
	$$
	\bigcap_{n \in P} f_n^{-1}(-\infty,a) \cap  \bigcap_{n \in M} f_n^{-1}(b,\infty) \neq \emptyset
	$$
	for all finite disjoint subsets $P, M$ of $\N$.
\end{defin}

\begin{defin} \label{d:tameF} \cite{GM-tame} 
We say that a bounded family $F$ of real valued (not necessarily, continuous)  
functions on a set $X$ is {\it tame} if $F$ does not contain an independent sequence.
\end{defin}

The following useful result combines some known facts.
It is based on results of Rosenthal \cite{Ro}, Talagrand \cite[Theorem 14.1.7]{Tal} and van Dulst \cite{Dulst}.

\begin{thm} \label{f:sub-fr}
Let $X$ be a compact space and $F \subset C(X)$ a bounded subset.
The following conditions are equivalent:
\begin{enumerate}
\item  $F$ is a tame family. 
\item
$F$ does not contain a subsequence equivalent to the unit basis of $l_1$.
\item
$F$ is a Rosenthal family for $X$.
\end{enumerate}
\end{thm}

\begin{thm} \label{Tamecriterion} 	
Let $X$ be a compact $S$-space. 
Suppose there is a point separating bounded $S$-invariant family $F$ of continuous 
real valued functions on $X$ such that $F$ is a tame family.
Then $(S,X)$ is a tame system. 
\end{thm}
\begin{proof}
Let $p$ be an arbitrary element of the enveloping semigroup $E(S,X)$.
Let $\mathcal{F} = \cls F$ be the pointwise closure of $F$ in $\R^X$. 
Then every element of $\mathcal{F}$ is fragmented by Theorem \ref{f:sub-fr}. Clearly 
$\{f \circ p : f \in F\}\subset \mathcal{F}$.  As $F$ separates points on $X$ this, in turn, easily implies (see, \cite[Lemma 2.3.3]{GM-rose}) 
that $p$ is a fragmented map. Thus $(S,X)$ is tame by a characterization of tame systems mentioned in Section \ref{s:classes}.
\end{proof}


In fact, we have the following sharper 
statement (see \cite[Theorem 6.5]{GM-rose}, and with more details \cite[Theorem 3.12]{GM-tame}):

\begin{thm} \label{WRNcriterion} 
Let $X$ be a compact $S$-space. The following conditions are equivalent:
	\begin{enumerate}
		\item $(S,X)$ is $\mathrm{WRN}$, that is, Rosenthal representable. 
		\item there exists a point separating bounded $S$-invariant family $F$ of continuous real functions on $X$ such that $F$ is a tame family (equivalently, Rosenthal family).
	\end{enumerate}
\end{thm}

\br	
	
\section{Every c-ordered system is Rosenthal representable} \label{Sec-Ros}  

\subsection{Families of functions with bounded variation are tame} 

\begin{defin} \label{d:BV} \ 
	\ben 
	\item 
	Let $(X,<)$ be a linearly ordered set and $(M,d)$ is a metric space. 
	We say that a bounded function $f: (X,<) \to (M,d)$ has variation not greater than $r$ (notation: $f \in BV_r)$ if 
	\begin{equation}  \label{eq1} 
		\sum_{i=1}^{n-1} d(f(x_i),f(x_{i+1})) \leq r
	\end{equation}

	for every choice of $x_1 \leq x_2 \leq \cdots \leq x_n$ in $X$.
		\item 
	For circularly ordered sets $(X,R)$ (instead of $(X,<)$) the definition is similar but we take \emph{cycles} $x_1, x_2, \cdots, x_n$ in $X$  (Definition \ref{d:cycl}) and require that 
	\begin{equation}  \label{eq2} 
	\sum_{i=1}^{n} d(f(x_i),f(x_{i+1})) \leq r
	\end{equation} 
	where $x_{n+1}=x_1$.    
\een 
\sk 
	The least upper bound of all such possible sums is the  {\it variation} of $f$; notation (both cases):   $\Upsilon(f)$. If  $\Upsilon(f) \leq r$ then we write $f \in BV_r(X,M)$ or, simply $BV_r(X)$ if $(M,d)=\R$. If $f(X) \subset [c,d]$ for some reals $c \leq d$ then we write also $f \in BV_r(X,[c,d])$. 
\end{defin}

\begin{remarks} \label{r:FinCol} \ 
	\ben 
	\item In the case of a linearly ordered set $(X,<)$ 
	denote by $M_+(X,[c,d])$ the set of all order-preserving functions $X \to [c,d]$. 
	Then $M_+(X,[c,d]) \subset BV_r(X,[c,d])$ for every $r \geq d-c$. In particular, $M_+(X,[0,1]) \subset BV_1(X,[0,1])$.  
	\item 
	For every finite interval partition of a c-ordered set $(X,R)$, every finite coloring 
	$f: X \to \Delta \subset \R$  of this partition is a function with bounded variation.   
	\item Note that the sum in Equation \ref{eq2} remains the same under the standard cyclic translation ($+1 \pmod n$). Also, one may reduce the computations in Definition \ref{d:BV}.2 to the injective cycles. 
	\item 
	 Let $q: X(c) \to X$ be the natural quotient (from Proposition \ref{cover}) for some $c \in X$, where 
	 $X(c)$ carries the natural linear order,  
	 and let $f: X(c) \to M$, $f_0: X \to M$ be functions such that $f= f_0 \circ q$. 
	 Then (using Remark \ref{r:FinCol}.3) we have 
	 $$\Upsilon(f) \leq  \Upsilon(f_0) \leq  \Upsilon(f)+\diam M.$$ 
	\item For every COP map $f_1: X \to Y$ and every $f_2 \in BV_r(Y,M)$ we have $f_2 \circ f_1 \in BV_r(X,M)$. 
	
	\een
\end{remarks}

\begin{lem} \label{np} \cite{Me-Helly} 
	For every linearly ordered set $(X,<)$ the set $BV_r(X,[c,d])$ is a tame family of functions. 
	In particular, this	is true also for $M_+(X,[c,d])$. 
\end{lem}

Let $(\T,R_{\a})$ be the cascade generated by an irrational rotation $R_{\a}$ of the circle $\T$.
Let $f: =\chi_D: \T \to \{0,1\}$ be the (discontinuous) characteristic function of the arc 
 $D=[a,a+s)\subset \T$. Consider the $\Z$-orbit $F$ of this function induced by the cascade  $(\T,R_{\a})$. Then $F$ is a tame family of (discontinuous) functions on $\T$. 
Much more generally we have: 

\begin{thm} \label{X(c)} 	
	Let $(X, R)$ be 
	a c-ordered set. Then any family of functions 
	$\{f_i: X \to [c,d]\}_{i \in I}$ 
	with finite total variation is tame. 
\end{thm}
\begin{proof} 
	It suffices to show that $BV_r(X,[c,d])$ is tame for every $r>0$.  
	The case of a c-ordered $(X, R)$ can be reduced to the linearly ordered cut space $X(c)$, 
	where we have Lemma \ref{np}, using Remark \ref{r:FinCol}.4 and 
the observation that the family $\{f_i\}$ is tame iff so is $\{f_i \circ q\}$.
\end{proof}

%
%
%
%

\sk

We are now ready to prove: 

\begin{thm} \label{t:CoisWRN} 
	Every c-ordered compact, not necessarily metrizable, $S$-space $X$ is Rosenthal representable 
	(that is, $\mathrm{WRN}$), hence, in particular, tame. 
	So, $\mathrm{CODS} \subset \mathrm{WRN} \subset \mathrm{Tame}$. 
\end{thm}

\begin{proof}
Let $X$ be a c-ordered compact $S$-system. We have to show that the $S$-system $X$ is $\mathrm{WRN}$.
By Theorem \ref{WRNcriterion}, this is equivalent to showing that there exists a point separating bounded $S$-invariant family $F$ of continuous real valued functions on $X$ such that $F$ 
is tame. By Theorem \ref{X(c)} bounded total variation of $F$ is a sufficient condition for its tameness. 
	
	Let $a \neq b$ in $X$. We can assume that $X$ is infinite. Take some third point $c \in X$. As in Proposition \ref{cover} consider the cut at $c$ where $c$ becomes the minimal element. 
	We get a compact linearly ordered set $X(c)=[c^-,c^+]$ and a natural quotient map 
	$$q: X(c) \to X, \ \ q(c^{-})=q(c^{+})=c$$ 
	and $q(x)=x$ in other points.  
	
	We have two similar cases: 
	
	1) $c^- < a < b < c^+$ 
	
	2)  $c^- < b < a < c^+$ 
	
	We explain the proof only for the first case.

	Since $X(c)=[c^-,a] \cup [a,b] \cup [b,c^+]$ is a linearly ordered compact space, its closed intervals $[a,b]$ and $[b,c^+]$  are also compact $\mathrm{LOTS}$. 
	By Nachbin's results \cite{Nach} continuous linear order preserving maps from compact LOTS to $[0,1]$ separate the points. Therefore, 
	one may choose continuous maps 
	$$f_1: [c^{-},a] \to [0,1], \ f_2: [a,b] \to [0,1], \ \ f_3: [b,c^+] \to [0,1]$$ such that $f_1$ is identically zero, $f_2$ and $f_3$ are order preserving, and 
	$$f_2(b)=f_3(b)=1, \ f_2(a)=f_3(c^+)=0.$$
	These three functions define a continuous function $f: [c^-,c^+] \to [0,1]$. It is easy to see that $f$ has 
	 total variation not greater than 2, that is, $f \in BV_2(X(c))$ 
and, clearly,  $0=f(a) \neq f(b)=1$.  
	
The  factor-function $f_0: X \to [0,1]$ (with $q(f_0(x))=f(x)$) is continuous because 
$q: X(c) \to X$ is a quotient map and $f$ is continuous. 
Moreover, by Remark \ref{r:FinCol}.4 we have $$\Upsilon(f) \leq  \Upsilon(f_0) \leq  \Upsilon(f)+1.$$
Thus, $\Upsilon(f_0) \leq 3$. Then, $\Upsilon(f_0 s) \leq 3$ (by Remark \ref{r:FinCol}.5) for every $s \in S$, because every $s$-translation preserves the c-order.  
Define $F:=F_0 S$, where $F_0$ is a set of continuous functions $X \to [0,1]$ with variation $\leq 3$. Since $F_0$ separates the points of $X$, $F_0 S$ is the desired bounded point-separating family of continuous functions which is tame and $S$-invariant. Now apply Theorem \ref{WRNcriterion}. 
\end{proof}

\subsection{Some purely topological notes} 
\label{s:topology}

As a direct topological consequence of Theorem  \ref{t:CoisWRN} note that every compact COTS is WRN.  
For instance, the two arrows space $K$ is Rosenthal representable. At the same time, $K$ is not Asplund representable (that is not RN) by a result of Namioka \cite[Example 5.9]{N}.  
That is, $K \in \mathrm{WRN} \setminus \mathrm{RN}$.

In a recent paper  \cite{mc} Martinez-Cervantes shows that a continuous image of a WRN compact space need not be WRN. This answers a question from \cite{GM-tame}. Note that $\beta \N$ is not WRN, a result of Todor\u{c}evi\'{c} (see \cite{GM-tame}).  Another result from  \cite{mc}, shows that  the \emph{Talagrand's compact} is also not WRN. 

\subsection{Representations of topological groups on Rosenthal spaces} 
\label{s:TopGr} 

\begin{thm} \label{t:GrRepr} 
Topological group $H_+(X)$ (with compact open topology) is Rosenthal representable for every c-ordered compact space $X$. For example, it is true for $H_+(\T)$.  
\end{thm} 
\begin{proof} (See also \cite{GM-tame}) Let $G:=H_+(X)$ with its compact open topology. 
The dynamical $G$-system $X$ admits a representation $(h, \a)$
 $$
 h: G \to \Iso(V), \ \ \a: X \to B^*
 $$ 
 on a Rosenthal Banach space $V$ by Theorem \ref{t:CoisWRN}. Then the homomorphism 
 $$
 h^*: G \to \Iso(V), \  \ g \mapsto h(g^{-1})
 $$ is a topological group embedding because 
 the strong operator topology on $\Iso(V)^{op}$ is identical with the compact open topology inherited from the action of this group on the weak-star compact unit ball $(B^*,w^*)$ in the dual $V^*$.  
	\end{proof}

The Ellis compactification $j: G \to E(G,\T)$ of the Polish group $G=H_+(\T)$ is a topological embedding.
In fact, observe that
the compact open topology on $j(G) \subset C_+(\T, \T)$ coincides
with the pointwise topology.
This observation implies, by \cite[Remark 4.14]{GM-survey} that
$\mathrm{Tame}(G)$ separates points and closed subsets.

Although $G$ is representable on a (separable) Rosenthal Banach space,
we have $\Asp(G)=\{constants\}$ and therefore any Asplund representation of this group is trivial
(this situation is similar to the case of the group $H_+[0,1]$, \cite{GM-suc}).
Indeed, we have $\SUC(G)=\{constants\}$ by \cite[Corollary 11.6]{GM-suc} for $G=H_+(\T)$, and we
recall that for every topological group
$\Asp(G) \subset \SUC(G)$.

\sk
Theorem \ref{t:GrRepr} suggests the following:

\begin{defin} \label{d:pseudolinear}
	Let us say that a topological group $G$ is 
	\emph{c-orderly} if $G$ is a topological subgroup of $H_+(X)$ for some circularly ordered compact space $X$.
\end{defin}

Thus, every orderly topological group $G$ is Rosenthal representable.
Recall that it is unknown yet (see \cite{GM-AffComp,GM-survey}) whether every Polish group is Rosenthal representable.

\begin{remark} \label{r:reversing}  
If $H < G$ is a subgroup of finite index and (G,X) is a dynamical system,
then clearly $F E(H,X) = E(G,X)$, with $F$ a finite set of representatives
of cosets of $H$ in $G$. Thus $(G,X)$ is tame iff $(H,X)$ is tame.
In particular, we see that the action of $H_{\pm}(\T)=H(\T)$ on the 
circle $\T$ is tame. 
In fact, $H_{\pm}(X)=H(X)$ for every connected c-ordered compact space $X$, where
$H_{\pm}(X)$ is the group of all bijections $h: X \to X$ which are either c-order preserving or c-order reversing.
This follows from \cite[Theorem 14]{Kok}
\end{remark}

\subsection{When the universal system $M(G)$ is c-ordered ?} 

Recall that for every topological group $G$ there exists, the canonically defined, 
universal minimal system $M(G)$
and universal irreducible affine $G$-system $I\!A(G)$. 
See for example, \cite{Ellis, Gl-book1,PestBook,UspComp,KPT}. 
In \cite{GM-tame} we discuss some examples of Polish groups $G$, 
for which $M(G)$ and $I\!A(G)$ are tame. 
These properties can be viewed as natural generalizations of 
extreme amenability and amenability, respectively. 

\begin{question}
	Find examples where $M(G)$ is c-ordered (more generally, tame).
\end{question}

Let us say that $G$ is \emph{intrinsically c-ordered} (\emph{intrinsically tame}) if the $G$-system $M(G)$ is c-ordered (respectively, tame). 
In particular, we see that $G=H_+(\T)$ 
is intrinsically c-ordered, using a well known result of Pestov \cite{Pest98} which identifies $M(G)$ as the tautological action of $G$ on the circle $\T$. 
Note also that 
the Polish groups $\Aut(\mathbf{S}(2))$ and 
$\Aut(\mathbf{S}(3))$, of automorphisms 
of the circular directed graphs $\mathbf{S}(2)$ and $\mathbf{S}(3)$, are also intrinsically c-ordered. 
The universal minimal $G$-systems for the groups $\Aut(\mathbf{S}(2))$ and 
$\Aut(\mathbf{S}(3))$ are computed in \cite{van-the}. One can show that $M(G)$ for these groups are c-ordered, see \cite{GM-tame}.

\br

\section{Some Sturmian like symbolic $\Z^k$-systems are c-ordered} 
\label{s:m}

We will see that several coding functions 
come from c-ordered systems (in particular, are tame) including some
multidimensional analogues of Sturmian sequences.
The latter are defined on the groups $\Z^k$ and instead of the characteristic function
$f:=\chi_D$ (with $D=[0,c)$) we consider finite coloring of the space leading to shifts with finite alphabet.
\begin{defin}  \label{d:ColFun} 
	Let $X$ be a c-ordered set and $c_0, c_1, \dots, c_{d}$ be a cycle of distinct elements. 
	Consider a finite cyclic partition
	$$
	X=\cup_{i=0}^{d} [c_i,c_{i+1}),
	$$	
	where $c_{d+1}=c_0$. 
	We say that 
	$f: X \to \Delta:=\{0, 1, \dots, d\}$ is a \emph{coloring function} if $f$ is constant on 
	each arc $[c_i,c_{i+1})$.  
	If, in addition, different arcs have different colors (equivalently, $f$ is onto) we say that $f$ is \emph{proper}.
	For instance, the following standard coloring function  is proper (and a COP map with $\Delta=C_{d+1}$). 
	$$
	f: X \to \Delta:=\{0, \dots ,d\}, \ \ \ f(t)=i \ \text{iff} \ t \in [c_i,c_{i+1}).
	$$
\end{defin}

\begin{defin} \label{d:ColFunT} 
	For $X:=\T$ consider a cyclic partition $\T=\cup_{i=0}^{d} [c_i,c_{i+1})$ and a proper coloring function $f: \T \to \Delta$. For a given $k$-tuple
	$(\a_1,\dots ,\a_k) \in \T^k$ of angles, where at least one of them is irrational,
	define the homomorphism 
	$$
	h: \Z^k \to \T, \ \ \ s:=(n_1, \dots,n_k) \mapsto n_1 \a_1 + \dots + n_k \a_k,
	$$
	and the induced c-order preserving action $\Z^k \times \T \to \T, \ (s,t) \mapsto t+ h(s)$. 
	Given a point $z \in \T$ consider the corresponding coding function
	$$
	\varphi=m(f,z): \Z^k \to \{0, \dots ,d\} \ \ \ (n_1, \dots,n_k) \mapsto f(z + n_1 \a_1 + \cdots + n_k \a_k).
	$$
	We call such a sequence a \emph{multidimensional 	$(k,d)$-Sturmian like sequence}. 
	
\end{defin}

%
%

\sk

The case of $k=1$ is studied in \cite{Pikula, Auj}.  
For some multidimensional Sturmian like $\Z^k$-systems see, for example, \cite{BFZ,Fernique}. 
Consider the associated symbolic Sturmian like $\Z^k$-system $G_{\varphi}  \subset  \Delta^{\Z^k}$ (from Definition \ref{d:tametype1}). 
The following theorem yields many examples of c-ordered (in particular, tame) symbolic $\Z^k$-systems and tame coding functions.


 \begin{thm} \label{t:multi} 
 	Let $\varphi=m(f,z)$ be the coding function induced by a 
 	proper coloring function $f:~\T \to \Delta$. 
 	Then the corresponding symbolic Sturmian like $\Z^k$-system $G_{\varphi} \subset  \{0,1, \cdots, d\}^{\Z^k}$ is isomorphic to a  c-ordered minimal metric $\Z^k$-system $\T_A$ from Lemma \ref{c-doubling} for some $A \subset \T$. 
 	Furthermore, $G_{\varphi}$ is a tame $\Z^k$-system and $\varphi \in \mathrm{Tame}(\Z^k)$. 
 \end{thm}   
 \begin{proof} 
 	First observe that the action of $G:=\Z^k$ on $\T$ 
 	$$
 	G \times \T \to \T, \ \ ((n_1, \dots,n_k),t) \mapsto t + n_1 \a_1 + \cdots + n_k \a_k.
 	$$
 	is c-ordered. By our assumption one of the the angles 
 	$\a_1,\dots ,\a_k \in [0,1)$ is irrational, say $\a_1$. 
 	Thus, the image $h(G)$ is dense in $\T$ and the action $G \times \T\to \T$ is minimal. 
 	
 	Let 
 	$E$ be the set of all endpoints $c_i$ of the given arcs from Definition \ref{d:ColFunT}.	
 	Let $A:=GE$. 
 	Consider the splitting points construction. 	
 	By Lemma \ref{c-doubling} we get a compact metrizable minimal c-ordered space $\T_A$, 
 	the  $G$-projection 
 	$	\nu: \T_A \to \T $, 
 	and the continuous action 
 	$$
 	G \times \T_A \to \T_A, \ g (s^+)=(gs)^+, 
 	g (s^-)=(gs)^-, g(x)=gx \ \forall s \in A, \ \forall x \notin A.
 	$$


 	As in Definition \ref{d:ColFunT} consider the  coding function 
 	$$\varphi:=m(f,z) : \Z^k \to \R, \ \ g \mapsto f(g(z))$$ 
 	and the corresponding subshift   
 	(Definition \ref{d:tametype1}.2)
 	$$
 	G_{\varphi} : = \cls_p (G\varphi) \subset  \{0,1, \cdots, d\}^{\Z^k}. 
 	$$ 	
 	Then $G_\varphi$ is a symbolic $\Z^k$-system and $\ga_1: G \to G_{\varphi}, \ g \mapsto g \varphi$ is a $G$-compactification. 
 	
 	Consider also the $G$-compactification  $\ga_2: G \to \T_A$, where $\ga_2$ is the (dense) orbit $G$-map 
 	$\ga_2(g) = g(z^+)$, if $z \in A$ and $\ga_2(g) =  g(z)$, if $z \notin A$.  
 	Our aim is to show that these $G$-compactifications $\ga_1$ and $\ga_2$ are $G$-isomorphic.

 	$G_{\varphi} \subset  \{0,1, \cdots, d\}^{\Z^k}$ is a pointwise compact subset of $C(G)$. 
 	Consider the (pointwise continuous) function 
 	$\widetilde{\varphi}: G_{\varphi} \to \R,  \ \omega \mapsto \omega(e).$
 	Then ${\varphi}=\widetilde{\varphi} \circ \ga_1$. 
 	Clearly, $\widetilde{{\varphi}}G$ separates points of $G_{\varphi}$. These facts easily yield  
 	(see also \cite[Proposition 2.4]{GM1}) that the $G$-subalgebra $\A_1 \subset \RUC(G)$ of the $G$-compactification $\ga_1: G \to G_{\varphi}$ is the least Banach unital $G$-subalgebra in $\RUC(G)$ (with respect to the right action of $G$) which contains 
 	$\varphi G$.

 	Using the given coloring function $f: \T \to \Delta$ define 
 	$$f^+: \T_A \to \Delta, \ \ f^+(t)=f(t) \ \ \ \forall \ t \in [c_{i-1}^+,c_{i}^+).$$
 	This function is continuous because each $[c_i^+,c_{i+1}^+)$ is a clopen subset of $\T_A$. 
 	Note that 
 	$\varphi: G \to \Delta$ comes from $\ga_2: G \to \T_A$. Namely,
 	$\varphi=f^+ \circ \ga_2$.  
 	Therefore, the $G$-subalgebra $\A_2 \subset \RUC(G)$ of the compactification $\ga_2$ contains ${\varphi}$ and hence also $\A_1$ (by its minimality property mentioned before). 
 	So, there exists a quotient $G$-map $q: \T_A \to G_{\varphi}$ such that $q \circ \ga_2=\ga_1$. 
 	We have to show that $q$ is isomorphism, equivalently, that we have also the converse inclusion $\A_2 \subseteq \A_1$. 
 	By the basic properties of compactifications 
 	it suffices to show now that the $G$-orbit $f^+G$ of 
 	$f^+: \T_A \to \R$ separates the points of $\T_A$. 	
 	%
 	%
 	%
 	%
 	Let $x,y \in \T_A$ be distinct points. Consider the $G$-map $\nu: \T_A \to \T $. We have the following two cases: 
 	
 	\sk
 	
 	(a)  (twins) $\nu(x)=\nu(y)$. 
 	
 	So, $x=c_i^-, y=c_i^+$, or  $x=c^+, y=c^-$. Then $f^+(c_i^-)=f(c_{i-1}), f^+(c_i^+)=f(c_{i})$ or 
 	$f^+(c_i^+)=f(c_{i}), f^+(c_i^-)=f(c_{i-1})$.  Since $f$ is proper, 
 	in both cases we have $f^+(x) \neq f^+(y)$.
 	
 	\sk
 	
 	(b) (non-twins) $x_0:=\nu(x) \neq y_0:=\nu(y)$. 
 	
 	Then since $\a_1$ is irrational there exists $n \in \Z$ such that the points 
 	$x_0+n \a_1$ and $y_0 + n \a_1$ belong to different arcs in the given partition. 
 	Since $f$ is proper, it follows that the function $$f^+g: \T_A \to \Delta, \ t \mapsto f^+(g(t))$$ 
 	with $g:=(n_1,0, \cdots,0) \in \Z^k$ 
 	separates the points $x,y$.   	
 	
 	So, $\ga_1$ and $\ga_2$ are $G$-isomorphic. In particular, $G_{\varphi}$ is a tame $G$-system 
 	(being c-ordered, Theorem \ref{t:CoisWRN}). 
 	Finally, $\varphi \in \mathrm{Tame}(\Z^k)$ because ${\varphi}=\widetilde{\varphi} \circ \ga_1$.
 \end{proof}



\begin{remarks} \label{r:afterThm} \ 
	\ben 
	\item In particular, every Strurmian like rotation bisequence (for $k=1, d=1$) from Definition \ref{d:StCode} satisfies the conditions of Theorem \ref{t:multi}.  
	For the case $k=1$ this result gives a new proof of a result in \cite{Masui}. 
	One may get in this way c-ordered subshifts $X \subset \{0,1\}^{\Z}$ which are not Sturmian (with complexity greater than $p(n)=n+1$).  
	
	\item 
	At least for the case of $k=1$, the tameness of the corresponding symbolic $\Z$-systems coming from coding functions of Definition \ref{d:ColFunT} can be proved also by results of Pikula \cite{Pikula} and Aujogue \cite{Auj}.  
	\item 
	Theorem \ref{t:multi} can be modified for a slightly more general partitions 
	$\T=\cup_{i=0}^{d} I_i$, 
	where each $I_i$ is an 
	arc on $\T$ (open, closed or containing one of the boundary points). 
	
 \een 
\end{remarks}

\br

 \section{Minimal c-ordered $\Z$-systems} 
 \label{s:MinC-Ord}

 \begin{thm}\label{m-c}
 	Let $(X,\sigma)$ be a  circularly ordered topologically transitive cascade
 	with no isolated points. 
 	Then $(X,\sigma)$ is minimal and there exists an irrational $\alpha \in\R$ and an 
 	$R_\al$-invariant subset $A \subset \T = \R/\Z$ such that:
 	$(X,\sigma) \cong \Sp(\T,R_\al; A)$ (defined in Lemma \ref{c-doubling} and Remark \ref{r:add}). 
 \end{thm}
 
 \begin{proof}
 	Let $C_X \subset X \times X \times X$ be the circular order on $X$.
 	Let 
 	$$
 	\Om = \{(x,x') \in X \times X : x =x', \ {\text or}\ 
 	(x,x')_C = \emptyset, 
 	\ {\text or}\ (x',x)_C = \emptyset \}.
 	$$
 	Note that if the pairs $(x, x')$ and $(x', x'')$ are in $\Omega$ then
 	necessarily $x'' = x$, as otherwise                                                                                                                                                                                   $x'$ would be an isolated point.
 	Now it is easy to check that $\Om$ is an ICER on $X$. 
 	
 	Let $Y = X/\Om$ be the quotient dynamical system and
 	let $\pi : X \to Y$ denote the quotient homomorphism.
 	We denote the corresponding transformation on $Y$ by $S$.
 	Thus $\pi(\sigma x) = S\pi(x)$ for every $x \in X$.
 	
 	It is easy to see that the circular order on $X$ induces a circular
 	order, $C_Y$, on $Y$ and that $\pi$ respects this ordering 
 	in the sense of Definition \ref{c-ordMaps}. 
 	
 	As $Y$ is compact it follows that $C_Y$ is a complete and
 	dense circular order. Since $Y$ is a minimal cascade it is in particular separable. 
 	This, in turn, implies that the countable collection of open
 	intervals defined via a dense countable subset forms
 	a countable basis for the topology on $Y$.
 	Therefore $Y$ (with its order topology)
 	is a connected metrizable compact space (a continuum).
 	Moreover, it has the property
 	that by omitting any two distinct points it becomes
 	disconnected. These properties characterize the circle 
 	$\T=\R / \Z$ (see, for example, \cite{Cech}).
 	
 	As $S : Y \to Y$ is topologically transitive it follows,
 	by a theorem of Poincar\'{e} \cite{Po}, that $S$ is conjugate to
 	an irrational rotation $R_\al : \T \to \T$, given by
 	$R_\al(y) = y + \al \pmod{1}$ for some irrational number $\al \in \R$.
 	We now identify $(Y,S) = (\T, R_\al)$. 
 	
 	It follows that $A$ is $R_\al$-invariant, where $A = \{y \in Y : |\pi^{-1}(y)| =2\}$,
 	and $\T \setminus A = \{y \in Y : |\pi^{-1}(y)| =1\}$.
 	In turn, this fact implies that $(X,\sigma)$ is also minimal. 
 	Moreover, by Lemma \ref{c-doubling}.3 we easily get that $(X,\sigma)$ is 
 	isomorphic to  $\Sp(\T,R_\al; A)$.	
 \end{proof}
 
\br 
 
 In \cite{Denj} Denjoy shows that a circle homeomorphism $\psi : \T \to \T$ is
 either minimal (in which case it is conjugate to a minimal rotation),
 or it admits a unique minimal set $X \subsetneq \T$ which is a Cantor set.
 In \cite{Ma} Markley gives a nice characterization of minimal cascades $(X,\sigma)$
 which can be embedded into a circle homeomorphism system $(\T, \psi)$.
 
 In the following theorem $\Delta =\{(x,x) : x \in X\}$ is the diagonal subset
 of $X \times X$, $P \subset X \times X$ is the proximal relation,
 $P'$ is the set of accumulation points of $P$, and $D \subset X \times X$
 is the {\em distal structure relation} of the system $(X,\sigma)$; i.e.
 $D$ is the smallest ICER such that $X/D$ is distal (see \cite{EG}).

 \begin{thm}[Markley]
 	Let $(X, \sigma)$ be a minimal cascade, where $X$ is a compact Hausdorff space. 
 	Then $(X, \sigma)$ can be imbedded in some $(\T, \psi)$, where $\psi$ has no periodic points,
 	if and only if $P' \subset \Delta$, $P \setminus \Delta$ is countable,
 	and $X/D$ is homeomorphic to $\T$.
 \end{thm}
 
 It also follows from his proof that when $(X,\sigma)$ satisfies the assumptions of the theorem
 then $P = D$ and the cardinality of $P[x]$ is one or two for each $x \in X$. Thus the 
 homomorphism $\pi : X \to X/D =\T$ is either an isomorphism or 
 it is one-to-one on the complement of the preimage of a countable set in $\T$, and two-to-one
 on the preimage of that set.
 Every such cascade, being a compact system topologically embeddable in $(\T,\psi)$, inherits a topologically compatible circular order. 
%
\br

 Combining these results we get:
 
 \begin{thm} \label{t:TFAE}
 	Let $(X,\sigma)$ be a minimal cascade with $X$ compact Hausdorff.
 	The following conditions are equivalent.
 	\begin{enumerate}
 		\item
 		$(X,\sigma)$ is an infinite, circularly ordered, minimal cascade
 		and $X$ is second countable.
 		\item
 		$(X, \sigma)$ can be imbedded in some $(\T, \psi)$, 
 		where $\psi$ is a homeomorphism of $\T$ which has no periodic points.
 		\item
 		$(X,\sigma)$ satisfies Markley's conditions (namely,  $P' \subset \Delta$, $P \setminus \Delta$ is countable,
 		and $X/D$ is homeomorphic to $\T$).
 		\item
 		There exists an irrational number $\al \in \R$ such that 
 		$(X,\sigma) \cong \Sp(\T,R_\al; A)$ for some countable, $R_\al$-invariant subset $A \subset \T$.
 	\end{enumerate}
 \end{thm}
 
 \begin{proof}
 	Markley's theorem asserts that (2) and (3) are equivalent.
 	From his proof it follows that (3) implies (4). 
 	It follows from Theorem \ref{m-c} that (1) and (4) are equivalent.
 	Finally, interpolating an interval between any pair 
 	of the countable collection $\{(a^+, a^-)\ : a \in A\}$ of split points in
 	$\Sp(\T, R_\al; A)$,  it is easy to see that the resulting space is
 	a circle $\T$. Then one can define a homeomorphism $\psi : \T \to \T$,
 	in such a way that the corresponding inclusion map  $i : X \to \T$
 	is an embedding of dynamical systems $i : (X,\sigma) \to (\T, \psi)$;
 	thus, showing that (4) implies (2).
 \end{proof}

 \begin{cor}
 	\begin{enumerate}
 		\item
 		To every 
 		infinite 
 		minimal circularly ordered cascade $(X,\sigma)$ corresponds a unique 
 		(irrational) rotation number $\al \in (0,1)$. 
 		\item
 		Given two minimal circularly ordered cascades $(X,T)$ and $(Y,S)$,
 		they either have rationally independent rotation numbers, in which case they are disjoint;
 		or their rotation numbers are rationally dependent and in this case
 		they admit a common extension of the form
 		$\Sp(\T, R_\al; A)$.
 	\end{enumerate}
 \end{cor}

 \subsection{Enveloping semigroups which are c-ordered} 
 \label{s:E}

 \br 
 
 We now turn to the study of the enveloping semigroup $E(\T,R_\al ; A)$ 
 of $\Sp(\T,R_\al; A)$. In \cite[Example 14.10]{GM1} we have shown that
 when $A$ is a single orbit of the system $(\T,R_\al)$, say $A = \{n \al : n \in \Z\}$, then
 $E= E(\T,R_\al ; A)$ can be identified with the 
 disjoint union 
 $ \T_{\T} \cup \{\sigma^n : n \in \Z\}$, where $(\T_{\T},\sigma)$ is Ellis' {\em double circle} cascade:
 $\T_{\T} = \{\beta^{\pm} : \beta \in \T = [0,1)\}$ and $\sigma \circ \beta^{\pm} = (\beta + \al)^{\pm}$. 
 $E$ becomes a circularly ordered cascade, where $E=\T_{\T} \cup \Z$ is a c-ordered subset of 
 the c-ordered lexicographic order $\T \times \{-,0,+\}$ (see Remark \ref{r:add}). 
 Under this definition for every $n \in \Z$ we have $[n\al^+, \sigma^n, n\al^-]$. 
 Since the interval $(n \alpha^+, n \alpha^-) \subset E$ contains only the single element $\sigma^{n}$ for every  $n \alpha \in G=\Z$ we get that every element of $G=j(G)$ is isolated in $E$. So, in this case $E =\T_{\T} \cup \Z$, where each point of $\Z$ is isolated in $E$.


 Since clearly every $\Sp(\T,R_\al; A)$ is a factor of $\T_{\T} \cup \Z$ we have the following.
 
 \begin{cor}
 	For every irrational $\al \in \R$ and every $R_\a$-invariant subset 
 	$A \subset \T$, the enveloping semigroup of $\Sp(\T, R_\al; A)$ is the
 	circularly ordered cascade $\T_{\T} \cup \Z$. It contains $\T_{\T}$ as its unique minimal left ideal; 
 	or equivalently as its unique minimal subset.
 \end{cor}
 
%
%


 \sk

 \begin{prop}
 	Every infinite, point transitive, circularly ordered cascade $(Y,S)$ is a factor of $\T_{\T} \cup \Z$ as above. 
 \end{prop}
 
 \begin{proof}
 	We sketch the straightforward proof. 
 	If $(Y,S)$ is minimal then by Theorem \ref{m-c} it is of the form $(\Sp(\T,R_\al; A)$ and we are done.
 	So we now assume that $(Y,S)$ is not minimal.
 	
 	Suppose $y_1, y_2, y_3$ are three distinct points
 	such that  $[y_1, y_2, y_3]$ and such that the intervals $(y_1, y_2)$
 	and $(y_2, y_3)$ are empty. Then $y_2$ is an isolated point and is therefore
 	necessarily on the unique orbit of transitive points.

 Let $\Om$ be 
 	 the smallest ICER which contains 
 	the pairs $(y,y') \in Y\times Y$
 	such that the interval $(y, y')$ 
 	 is empty.
 	It is not hard to see, using the remark above, that $Y/\Om$ is a circle on which $S$ 
 	induces a minimal homeomorphism.
 	Then $Y/\Om$ can be identified with $(\T, R_\al)$ for some irrational $\al$
 	and one checks that the quotient map $\pi : (Y,S) \to (\T,R_\al)$ is a proximal extension.
 	
 	This implies that $Y$ contains a unique minimal subset $Z \subset Y$,
 	and by Theorem \ref{m-c} $Z \cong \Sp(\T,R_\al, A)$ for some irrational $\al$
 	and an $R_\al$-invariant $A \subset \T$.
 	
 	Next one shows that for every $z \in A$ the corresponding open interval $(z^-, z^+)$
 	contains at most one point and moreover, there is a unique orbit $\{z + n \al : n \in \Z\} \subset A$
 	for which this happens; i.e. there are points $c_n$ with $[(z + n\al)^-, c_n , (z + \al)^+]$.
 	We  now identify $c_n$ with $S^n$. Finally, using the information we already have on
 	$(Y,S)$ it is clear how to define the factor map  $\T_{\T} \cup \Z \to (Y,S)$. 
 \end{proof}

 \subsection{General discrete group action}
 
 Let $G$ be an infinite countable group. 
 A $G$-system $(G,X)$ is {\em proximal} if for every pair $x, x' \in X$ and a every neighborhood $V$ in $X \times X$ of the diagonal 
 $\Del_X =\{(z, z) : z \in X\}$,
 there is $g \in G$ with $(gx, gx') \in V$.
 We say that  $(G,X)$ is {\em extremely proximal} if
 for every closed subset $A \subset X$ and every nonempty open 
 subset $U \subset X$ there exists an element $g \in G$
 with $g(A) \subset U$.
 Recall that the $G$-action on a metrizable $X$ is equicontinuous iff
 there is a compatible metric on $G$ with respect to which every 
 $g \in G$ acts as an isometry. We then say that the action is 
 {\em isometric}.
 
 We have the following neat dichotomy theorem of Malyutin \cite{Mal}:

 \begin{thm}
 	Every minimal system $(G,\mathbb{T})$, a continuous action of $G$ on the circle $\mathbb{T}$, 
 	is either conjugate to an isometric action on $\mathbb{T}$ or it is a finite to one extension of an extremely proximal action where the factor
 	map is a covering map.
 \end{thm}

 It is easy to see that a minimal proximal $G$-action of an abelian group
 $G$ is necessarily trivial. Thus it follows that every minimal action of abelian $G$ on $\mathbb{T}$ is isometric. An example of 
 a minimal proximal action is provided by the natural action of
 $SL(2,\Z)$ on the projective line $\mathbb{P}^1$, the collection
 of lines through the origin in $\R^2$
 (which is homeomorphic to $\mathbb{T}$).
 The action of $SL(2,\Z)$ on the space of rays emanating
 from the origin, 
 which is again homeomorphic to $\mathbb{T}$,
 is an example of a two-to-one covering of a proximal action.

 \begin{thm}\label{m-c2}
 	Let $(G,X)$ be an infinite circularly ordered minimal $G$-system.
 	Then there exists a minimal 
 	$G$-action on $\mathbb{T}$ and a 
 	$G$-invariant subset $A \subset \T$ such that:
 	$(G,X) \cong \Sp(\T,G; A)$. 
 \end{thm}
 
 \begin{proof}
 	Same as the (relevant parts of the) proof of Theorem \ref{m-c}.
 \end{proof}

 \bibliographystyle{amsplain}

\begin{thebibliography}{10}
	

	
	\bibitem{AH}
	E. Akin, K. Hrbacek, 
	\emph{Complete homogeneous LOTS}, Topology Proc. 26 (2001), 367-406. 
	
	\bibitem{Auj}
	J. B. Aujogue,
	\emph{Ellis enveloping semigroup for almost canonical model sets}, ArXiv:1305.0879, May, 2014.
	
	
	
%
	
	\bibitem{BFT} J. Bourgain, D.H. Fremlin and M. Talagrand,
	{\it Pointwise compact sets in Baire-measurable functions}, Amer.
	J. of Math., \textbf{100:4} (1977), 845-886.
	
	
	
	
	
	\bibitem{BFZ} 
	V. Berthe, S. Ferenczi, L.Q. Zamboni,
	\emph{Interactions between Dynamics, Arithmetics and
		Combinatorics: the Good, the Bad, and the Ugly}, Contemporary Math., {\bf 385} (2005), 333-364.
	
	
	
	
	\bibitem{Cech}
	E. \v{C}ech, \emph{Point Sets}, Academia, Prague, 1969.
	
	\bibitem{CJ}
	S. \v{C}ernak, J. \v{J}akub\v{i}k, \emph{Completion of a cyclically ordered group}, Czech. Math. J., 37 (1987), 157-174. 
	
	
	
\bibitem{Ch-Si}
A. Chernikov and P. Simon, 
{\em Definably amenable NIP groups},
arXiv:1502.04365v1 [math.LO] 15 Feb 2015.

	
	\bibitem{Denj}
	A. Denjoy,
	{\em Sur les courbes d\'{e}finies par les \'{e}quations diff\'{e}rentielles \`{a} la surface du tore}, 
	J. de Math. {\bf 11}, (1932),  333--75.
	
	\bibitem{Dulst}
	D. van Dulst, \emph{Characterizations of Banach spaces not containing $l^1$}.
	Centrum voor Wiskunde en Informatica, Amsterdam, 1989.
	
	\bibitem{Ellis}
	R. Ellis,
	{\em Lectures on Topological Dynamics\/},
	W. A. Benjamin, Inc.\ , New York, 1969.
	
	

	
	\bibitem{EG}
	R. Ellis and W. H. Gottschalk, 
	{\em Homomorphisms of transformation groups}, 
	Trans. Amer. Math. Soc. {\bf 94}, (1960), 258--71.
	
	\bibitem{Fernique}
	T. Fernique, \emph{Multi-dimensional Sturmian sequences and generalized substitutions},
	Int. J. Found. Comput. Sci., \textbf{17} (2006), pp. 575-600.
	
	
	
	
	
	\bibitem{Gl-book1}
	E. Glasner, \emph{Proximal flows,} Lect. Notes, 517, Springer, 1976.
	
	\bibitem{Gl-tame}
	E. Glasner, {\it On tame dynamical systems}, Colloq. Math.
	\textbf{105} (2006), 283-295.
	
	\bibitem{Gl-str} E. Glasner, {\it The structure of tame minimal
		dynamical systems}, Ergod. Th. and Dynam. Sys. {\bf 27} (2007),
	1819--1837.
	
	
	\bibitem{GM1}
	E. Glasner, M. Megrelishvili, \emph{Linear representations of
		hereditarily non-sensitive dynamical systems}, Colloq. Math.,
	\textbf{104} (2006), no. 2, 223-283.
	
	\bibitem{GM-suc}
	E. Glasner and M. Megrelishvili, {\it New algebras of functions on
		topological groups arising from $G$-spaces}, Fundamenta Math., \textbf{201}
	(2008), 1-51.
	
	
	
	\bibitem{GM-rose}
	E. Glasner, M. Megrelishvili, {\it Representations of dynamical
		systems on Banach spaces not containing $l_1$},
	Trans. Amer. Math. Soc.,  \textbf{364} (2012),  6395-6424.
	ArXiv e-print: 0803.2320.
	
	\bibitem{GM-AffComp}
	E. Glasner, M. Megrelishvili,
	\emph{Banach representations and affine compactifications of dynamical systems},
	in: Fields institute proceedings dedicated to the 2010 thematic program on asymptotic geometric analysis,
	M. Ludwig, V.D. Milman, V. Pestov, N. Tomczak-Jaegermann (Editors), Springer, New-York, 2013. 
	ArXiv version: 1204.0432.
	
	\bibitem{GM-survey}
	E. Glasner, M. Megrelishvili,
	\emph{Representations of dynamical systems on Banach spaces,}
	in: Recent Progress in General Topology III, 
	(Eds.: K.P. Hart, J. van Mill, P. Simon),  
	Springer-Verlag, Atlantis Press, 2014, 399-470. 
	
	\bibitem{GM-tame}
	E. Glasner, M. Megrelishvili, 
	\emph{Eventual nonsensitivity and tame dynamical systems}, arXiv:1405.2588, 2014. 
	
	
	\bibitem{GMU} E. Glasner, M. Megrelishvili and V.V. Uspenskij,
	\emph{On metrizable enveloping semigroups}, Israel J. of Math.
	\textbf{164} (2008), 317-332.
	
%
%
%
%
	
	
	\bibitem{H}
	W. Huang, {\em Tame systems and scrambled pairs under an abelian
		group action\/}, Ergod. Th. Dynam. Sys. {\bf 26} (2006),
	1549--1567.
	
	
	\bibitem{Hunt}
	E.V. Huntington, \emph{Sets of completely independent postulates for cyclic order}, Proc. Nat. Acad. Sci. USA 10 (1924). 74-78. 
	
	\bibitem{Ibar}
	T. Ibarlucia,
	\emph{The dynamical hierarchy for Roelcke precompact Polish groups}, 
	Israel J. of Math., to appear. 
	
	
	
	\bibitem{KPT}
	A.S. Kechris, V.G. Pestov, and S. Todor\u{c}evi\'{c},
	Fra\"iss\'e limits, Ramsey theory, and
	topological dynamics of automorphism groups, Geom. Funct. Anal. \textbf{15} (2005), no. 1, 106-189.
	
	
	
	\bibitem{KL05}
D. Kerr and H. Li,
{\it Dynamical entropy in Banach spaces}, Invent. Math.
{\bf 162}  (2005), 649-686.
	
	
	\bibitem{KL}
	D. Kerr and H. Li,
	{\it Independence in topological and $C^*$-dynamics}, Math. Ann.
	{\bf 338}  (2007), 869-926.
	
	\bibitem{Kok}
	H. Kok, \emph{Connected orderable spaces},
	Math. Centhre Tracts {\bf 49}, Mathematisch Centrum, Amsterdam, 1973.
	
	\bibitem{Koh}
	A. K\"{o}hler, {\em Enveloping semigrops for flows}, Proc.
	of the Royal Irish Academy, {\bf 95A} (1995), 179--191.
	
	
	
	
	
	\bibitem{Mal}
	A. V. Malyutin,
	{\em Classification of the group actions on the real line and circle},
	Algebra i analiz, Tom 19, (2007), N. 2.
	St. Petersburg Math. J. Vol 19 (2008), 
	No. 2,  279--296.
	
	
	\bibitem{Ma}
	N. G. Markley,
	{\em Homeomorphisms of the circle without periodic points},
	Proc. London Math. Soc, {\bf (3)} 20, (1970), 688-698.
	
	
	\bibitem{mc}
	G. Martinez-Cervantes, \emph{On weakly Radon-Nikodym compact spaces,} ArXiv 1509.05324, September, 2015.
	
	\bibitem{Masui} 
	K. Masui, \emph{Denjoy systems and substitutions}, Tokyo J. Math.
	{\bf 32:1} (2009), 33-53. 
	
	
	
	
	
	
	\bibitem{Me-Helly}
	M. Megrelishvili, \emph{A note on tameness of families having bounded variation}, ArXiv, 2014. 
	
	
	\bibitem{MH} 
	M. Morse, G. A. Hedlund, 
	\emph{Symbolic Dynamics II. Sturmian Trajectories},
	American J. of Math., Vol. 62, No. 1 (1940), pp. 1-42.  
	
	\bibitem{Nach}
	L. Nachbin, \emph{Topology and order,} Van Nostrand Math. Studies, Princeton, New Jersey, 1965.
	
	\bibitem{N}
	I. Namioka, {\em Radon-Nikod\'ym compact spaces and
		fragmentability\/}, Mathematika {\bfseries 34} (1987), 258-281.
	
	
	
	
	\bibitem{Pikula}
	R. Pikula,
	\emph{Enveloping semigroups of affine skew products and Sturmian-like systems}, Dissertation,
	The Ohio State University, 2009.
	
	\bibitem{Pest98}
	V.G. Pestov,
	{\em On free actions, minimal flows, and a problem by Ellis},
	Trans. Amer. Math. Soc., {\bf 350} (1998),  4149-4165.
	
	\bibitem{PestBook}
	V.G. Pestov, \emph{Dynamics of infinite-dimensional groups. The
		Ramsey-Dvoretzky-Milman phenomenon.} University Lecture Series, {\bf 40}. American Mathematical Society, Providence, RI, 2006.
	
	\bibitem{Po}
	H. Poincar\'{e}, 
	{\em Sur les courbes d\'{e}finies par les \'{e}quations diff\'{e}rentielles}, 
	J. Math. Pures Appl. {\bf (4)}, 1, (1885), 167--224.
	
	
	


\bibitem{Ro} H.P. Rosenthal,
\emph{A characterization of Banach spaces containing $l_1$}, Proc.
Nat. Acad. Sci. U.S.A., \textbf{71} (1974), 2411--2413.

	
	
	
	\bibitem{Rom} 
	A.V. Romanov, 
	\emph{Ergodic properties of discrete dynamical
		systems and enveloping semigroups}, 
	Ergod. Th. \& Dynam. Sys. {\bfseries 36} (2016), 198-214. 
	
\bibitem{Tal}
M. Talagrand, \emph{Pettis integral and measure theory,} Mem. AMS No.
\textbf{51}, 1984.
	
	
	
	\bibitem{TodBook}
	S. Todor\u{c}evi\'{c}, {\em Topics in topology}, Lecture Notes in
	Mathematics, {\bfseries 1652}, Springer-Verlag, 1997.
	
	
	
	\bibitem{UspComp}
	V.V. Uspenskij, {\em Compactifications of topological groups},
	Proceedings of the Ninth Prague Topological Symposium (Prague,
	August 19--25, 2001). Edited by P. Simon. Published April 2002 by
	Topology Atlas (electronic publication). Pp. 331-346,
	ArXiv:math.GN/0204144. 
	
	\bibitem{van-the}
	L. Nguyen van Th\'{e},
	{\em More on the Kechris-Pestov-Todorcevic correspondence:
		precompact expansions}, 
	Arxiv: 1201.1270v3.
	
	
	
\end{thebibliography}

\end{document}